\documentclass{amsart}
\usepackage{graphicx}
\usepackage{amsmath, amssymb, euscript,  enumerate}

\newtheorem{thm}{Theorem}[section]

\newtheorem{lem}[thm]{Lemma}
\newtheorem{prop}[thm]{Proposition}
\theoremstyle{definition}
\newtheorem{defn}[thm]{Definition}
\theoremstyle{remark}
\newtheorem{rem}[thm]{Remark}
\numberwithin{equation}{section}

\newcommand{\R}{\mathbb R}
\newcommand{\eps}{\varepsilon}

\newcommand{\Sh}{{\mathbb S}}
\begin{document}
\title{ On Monge-Amp\'ere equations with homogenous right hand side}
\author{Panagiota  Daskalopoulos$^*$}
\address{Department of Mathematics, Columbia University, New York}
\email{pdaskalo@math.columbia.edu}
\author{Ovidiu  Savin}
\address{Department of Mathematics, Columbia University, New York}
\email{savin@math.columbia.edu}

\thanks{$*:$ Partially supported
by NSF grant 0102252}

\begin{abstract} We study the regularity and  behavior at the origin  of solutions to the two-dimensional degenerate   Monge-Amp\'ere
equation $\det D^2 u=|x|^\alpha$,   with  $\alpha>-2$. We show that
when   $\alpha>0$ solutions   admit  only two
possible behaviors near the origin, radial and non-radial, which
in turn  implies   $C^{2, \delta}$ regularity. We also show that the
radial behavior is unstable. For $\alpha<0$ we prove that
solutions  admit  only the radial behavior near the origin.
\end{abstract}
\maketitle

\section{Introduction}

We consider the degenerate two dimensional  Monge-Amp\'ere equation
\begin{equation}{\label{eq}}
\det D^2 u=|x|^\alpha,  \qquad  x\in   B_1
\end{equation}
on the unit disc   $B_1=\{ \, |x| \leq 1 \}$ of   $\R^2$ and  in the range of exponents $\alpha > -2$. Our goal is to investigate the behavior of solutions $u$  near the  origin,
where the equation becomes degenerate.

\smallskip

The study of \eqref{eq}  is motivated by the Weyl problem with nonnegative curvature,
posed in 1916 by Weyl himself: {\em Given a  Riemannian metric $g$ on the 2-sphere
$\Sh^2$ whose Gauss curvature is everywhere positive, does  there
exist a global $C^2$ isometric embedding $X: (\Sh^2, g) \to
(\R^3, ds^2)$, where $ds^2$ is the standard flat metric on $\R^3$?}

H. Lewy \cite{Le} solved the problem under the assumption that  the metric $g$
is analytic.
The solution to the Weyl problem, under the regularity assumption that $g$ has continuous
fourth order derivatives, was given in 1953 by L. Nirenberg  \cite{Ni1}.

P. Guan and Y.Y.  Li   \cite{G-L}  considered the question:  {\em
If the Gauss curvature of the metric  $g$ is nonnegative instead
of strictly positive and $g$ is smooth, is it still possible to
have a smooth isometric embedding ? }

It was shown  in  \cite{G-L} that  for any $C^4$-Riemannian metric $g$
on $\Sh^2$ with nonnegative Gaussian  curvature, there is always a $C^{1,1}$
global isometric embedding into $(\R^3, ds^2)$.

Examples   show that for some analytic metrics with positive Gauss
curvature on $\Sh^2$ except at one point, there exists only a
$C^{2,1}$ but not a $C^3$ global isometric embedding into $(\R^3,
ds^2 )$. Note that the phenomenon is global, since C.S. Lin
\cite{Lin} has shown that for any smooth 2-dimensional Riemannian
metric with nonnegative Gauss curvature there exists a smooth {\em
local} isometric embedding into $(\R^3, ds^2)$.

This leads to the following question, which was posed in \cite{G-L}: {\em  Under what conditions on a smooth metric $g$  on
$\Sh^2$ with nonnegative Gauss curvature, there is a
$C^{2,\alpha}$ global isometric embedding into $(\R^3, ds^2 )$,
for some $\alpha >0$, or even a $C^{2,1}$ ?}

The problem can be reduced to a partial differential equation of
Monge-Amp\'ere type that becomes degenerate at the points where
the Gauss curvature vanishes. It is well known that in general one
may have solutions to degenerate  Monge-Amp\'ere equations which
are at most $C^{1,1}$.

One may  consider a smooth Riemannian metric $g$  on $\Sh^2$ with
nonnegative Gauss curvature,  which has only one non-degenerate
zero. In this case, if we represent the $C^{1,1}$ embedding as a
graph, answering  the above question  amounts  to studying the
regularity at the origin of the degenerate  Monge-Amp\'ere
equation
\begin{equation}\label{dMA}
\det D^2 u = f, \qquad \mbox{on} \,\,  B_1
\end{equation}
in the case where the forcing term $f$ vanishes quadratically at
$x=0$. More precisely, it suffices to assume that $f(x) =|x|^2g(x)$,
where $g$ is a positive Lipschitz function.
 This leads to equation \eqref{eq} when $\alpha =2$.

\smallskip

In addition to the results mentioned above, degenerate equations of the form \eqref{dMA} on $\R^2$  were previously  considered by P. Guan in  \cite{PG1} in the case where $f \in C^\infty(B_1)$ and
\begin{equation}\label{eqn-guan}
 A^{-1} \, (x_1^{2l} + B \, x_2^{2m} ) \leq f(x_1,x_2) \leq  A \, (x_1^{2l} + B \, x_2^{2m} )
\end{equation}
for some constants $A >0, B \geq 0$ and positive integers $l \leq
m$. The  $C^\infty$ regularity of the solution $u$ of \eqref{dMA}
was shown in \cite{PG1}, under the additional condition that
$u_{x_2x_2} \geq C_0 >0$. It was conjectured in \cite{PG1} that
the same result must be true under the weaker condition that
$\Delta u \geq C_0 >0$. This was recently shown by P. Guan and I.
Sawyer in \cite{GS}.

\smallskip

Equation (\ref{eq}) has also an interpretation in the language of
optimal transportation with quadratic cost $c(x,y)=|x-y|^2$. In
this setting the problem consists in transporting the density
$|x|^\alpha\,  dx$ from a domain $\Omega_x$ into the uniform density
$dy$ in the domain $\Omega_y$ in such a way that we minimize the total
``transport cost", namely
$$\int_{\Omega_x}|y(x)-x|^2|x|^\alpha dx.$$
Then, by a theorem of Y. Brenier \cite{B}, the optimal map $x
\mapsto y(x)$ is given by the gradient of a solution of  the
Monge-Amp\'ere equation (\ref{eq}). The behavior of these
solutions at the origin gives information on the geometry of the
optimal map near the singularity of the measure $|x|^\alpha\, dx$.

\medskip

We will next state the results of this paper. We assume that $u$ is a solution of
equation \eqref{eq}. Then, $u$ is $C^\infty$-smooth away from the origin. The
following results describe the regularity of $u$ at the origin.
We  begin with the case when $\alpha >0$.

\begin{thm}{\label{t1}}
If $\alpha >0$, then $u\in C^{2, \delta}$ for a small $\delta$ depending
on $\alpha$.
\end{thm}

\smallskip
Theorem \ref{t1} is a consequence of Theorem \ref{t2} which shows
that there are exactly two types of behaviors near the origin.

\begin{thm}{\label{t2}}
If $\alpha>0$, and
\begin{equation}{\label{00}}
u(0)=0, \quad \nabla u(0)=0
\end{equation}
then, there exist positive constants $c(\alpha)$, $C(\alpha)$ depending
on $\alpha$ such that either $u$ has the  radial behavior
\begin{equation}{\label{rad}}
c(\alpha)|x|^{2+\frac{\alpha}{2}} \le u(x) \le
C(\alpha)|x|^{2+\frac{\alpha}{2}}
\end{equation}
or, in an appropriate system of coordinates, the  non-radial behavior
\begin{equation}{\label{nonrad}}
u(x)=\frac{a}{(\alpha+2)(\alpha+1)}|x_1|^{2+\alpha}+\frac{1}{2a}x_2^2
+O \left ( (|x_1|^{2+\alpha}+x_2^2)^{1+\delta} \right)
\end{equation}
for some $a>0$.
\end{thm}

The non-radial behavior \eqref{nonrad} was first shown by P. Guan
in \cite{PG1}, under the condition that $u_{x_2x_2} \geq C_0 >0$
near the origin,  and was recently generalized in  \cite{GS} to
only assume  that $\Delta u \geq C_0 >0$.

\smallskip
The next result states that  the radial behavior is unstable.

\begin{thm}{\label{t3}}
Suppose $\alpha>0$, let $u_0$ be the radial solution to
$(\ref{eq})$,
$$u_0(x)=c_\alpha|x|^{2+\frac{\alpha}{2}}$$
and consider the Dirichlet problem
$$\det D^2u=|x|^\alpha, \quad  \quad u=u_0-\eps \cos(2 \theta) \mbox{ on
$\partial B_1$}.$$ Then $u-u(0)$ has the nonradial behavior
$(\ref{nonrad})$ for small $\eps$.
\end{thm}

\smallskip
Subsequences of blow up solutions satisfying
(\ref{rad}) converge to homogenous solutions, as shown next.

\begin{thm}{\label{t4}}
Under the assumptions of Theorem $\ref{t2}$, if   $u$ satisfies
$(\ref{rad})$, then for  any sequence of $r_k \to 0$ the blow up
solutions
$$r_k^{-2-\frac{\alpha}{2}} u (r_k x)$$
have a subsequence that converges uniformly on compact sets to a
homogenous solution of $(\ref{eq})$.
\end{thm}

\smallskip

In the case $-2<\alpha<0$ solutions have only the radial
behavior.  Actually, we prove a stronger result by showing that $u$
converges to the radial solution $u_0$ in the following sense.

\begin{thm}{\label{t5}}

If $-2<\alpha<0$ and $(\ref{00})$ holds, then

$$\lim_{x \to 0} \frac{u(x)}{u_0(x)}=1.$$
\end{thm}

\medskip

Our results are based on the following argument: assume that a
section of $u$, say $\{u<1\}$, is ``much longer" in the $x_1$
direction compared to the $x_2$ direction. If $v$ is an affine
rescaling of $u$ so that $\{v<1\}$ is comparable to a ball, then
$v$ is an approximate solution of
$$\det D^2v(x) \approx c|x_1|^\alpha.$$
Hence, the geometry of small sections of solutions of  this new
equation provides information on the behavior of the small
sections of $u$. For example,  if the sections of $v$ are ``much
longer" in the $x_1$ direction (case $\alpha>0$) then the
corresponding sections of $u$ degenerate more and more in this
direction, producing the non-radial behavior \eqref{nonrad}. If
the sections of $v$ are longer in the $x_2$ direction (case
$\alpha<0$) then the sections of $u$ tend to become round and we
end up with a radial behavior near the origin.

\medskip

We close the introduction with the following remarks.

\begin{rem}
From the proofs one can see that the theorems above, with the
exception of the instability result, are still valid for the
equation with more general right hand side
$$\det D^2u=|x|^\alpha g(x)$$
with $g \in C^\delta(B_1)$, $g>0$.

\end{rem}

\begin{rem}  \textit{i.} We will show in the proof of Theorem \ref{t1}  that solutions of \eqref{eq},  with $\alpha >0$,  which satisfy  the radial behavior \eqref{rad} at the origin are of class $C^{2,\frac \alpha 2}$.
 \textit{ii.} Theorems \ref{t1},  \ref{t2} and the results of Guan in \cite{PG1}
and Guan and Sawyer in \cite{GS} imply that solutions of  \eqref{eq},
with $\alpha$ a positive integer,  which satisfy  the non-radial behavior \eqref{nonrad} at the origin are $C^\infty$-smooth.

\end{rem}

\begin{rem}
Equations of the form
\begin{equation}\label{eqn-gw}
\det D^2 w=|\nabla w|^{\beta}, \qquad  \beta = - \alpha
\end{equation}
for which the set $\{\nabla w =0\}$ is compactly included in the domain of
definition,
can be reduced to (\ref{eq}) by defining $u$ to be the Legendre transform
of $w$. Hence, Theorem \ref{t5} establishes the
sharp regularity of solutions $w$ of equation \eqref{eqn-gw} when $0 < \beta < 2.$
\end{rem}


\medskip
The paper is organized as follows. In Section \ref{sec0} we introduce
tools and notation to be used later in the paper. In Section \ref{sec1} we prove
Theorem \ref{t2}. In Section \ref{sec3} we establish the radial behavior of solutions
when $-2 < \alpha<0$, showing Theorem \ref{t5}. In Section \ref{sec4} we investigate
homogenous solutions and give the proof of Theorem \ref{t4}.  In Section \ref{sec5} we prove Theorem \ref{t3}.  Finally, in
Section \ref{sec2} we show that Theorem \ref{t2} implies Theorem \ref{t1}.

\medskip

\noindent{\em Acknowledgment:} We are grateful to P. Guan and  Y.Y. Li  for introducing us to this problem and for many useful discussions.

\section{Preliminaries}{\label{sec0}}

In this section  we investigate the geometry of the sections of
$u$, namely the sets
$$S_{t,{x_0}}^u:=\{u(x) < u(x_0) + \nabla u(x_0) \cdot (x-x_0) +t \}.$$
We omit the indices $u$ and $x_0$ whenever there is no possibility
of confusion. We recall some facts about such sections.

{\em John's lemma} (c.f. Theorem 1.8.2 in \cite{Gu})
states that any bounded convex set $\Omega \subset
\mathbb{R}^n$ is balanced with
respect to its center of mass. That is, if $\Omega$ has center of
mass at the origin, there exists an ellipsoid $E$ (with center of
mass $0$) such that
$$E \subset \Omega \subset k(n)E$$
for a constant $k(n)$ depending only on the dimension $n$.

Sections $S_{t,{x_0}}^u$ of solutions to Monge-Amp\'ere equations with
doubling measure $\mu$ on the right hand side also satisfy a balanced
property with respect to $x_0$.
We recall the following definition.

\begin{defn} [Doubling measure] The measure $\mu$ is doubling with respect to ellipsoids
in $\Omega$ if there exists a constant $c>0$ such that for any
point $x_0 \in \Omega$ and any ellipsoid $x_0+E \subset \Omega$
\begin{equation}{\label{doubling}}
\mu(x_0+E)\ge c \mu \left ((x_0+2E)\cap \Omega \right).
\end{equation}
\end{defn}

The following theorem, due to L. Caffarelli \cite{Ca1} holds.

\begin{thm} [Caffarelli]{\label{caf1}}
Let $u:\Omega \to \mathbb{R}$ be a (Alexandrov) solution of
$$ \det D^2u=\mu$$
with $\mu$ a doubling measure. Then, for each $S_{t,{x_0}} \subset \Omega$
there exists a unimodular  matrix $A_t$ such that
\begin{equation}{\label{1}}
k_0^{-1}A_t B_r \subset S_{t,{x_0}}-x_0 \subset k_0 A_t B_r
\end{equation}
with
$$r=t\, ( \mu(S_{t,{x_0}}))^{-1/n}, \qquad \det A_t=1.$$
for a constant $k_0(c,n)>0$.
\end{thm}

The ellipsoid $E=A_t B_r$ remains invariant if we replace $A_t$ with $A_t\, O$
with $O$ orthogonal, thus we may  assume that $A$ is triangular. If (\ref{1}) is
satisfied we write
$$S_t \sim A_t$$
and say that the eccentricity of $S_t$ is proportional to $|A_t|$.

\smallskip
The measure that appears in (\ref{eq}), namely
$$\mu:=|x|^{\alpha}\, dx$$
is clearly doubling with respect to ellipsoids for $\alpha>0$. We will
see in Section \ref{sec3} that this property is still true for $-1 <
\alpha<0$ but fails for $-2 < \alpha \le -1$.

\medskip
Next we discuss the case when the right hand side in the Monge-Amp\'ere
equation depends only on one variable, i.e
\begin{equation}\label{eqn-x1}
\det D^2 u=h(x_1).
\end{equation}
We will show in Section \ref{sec1} that such equations are
satisfied by blow up limits of solutions to $\det D^2 u =
|x|^\alpha$ at the origin, when $\alpha >0$.

These  equations remain invariant under affine transformations. Also, by
taking derivatives along the $x_2$ direction one obtains the  Pogorelov type estimate
$$u_{22}\leq C$$  in the interior of the sections of $u$.

Assume that $u$ satisfies  equation \eqref{eqn-x1}   in $B_1 \subset \R^n$, in any dimension $n \geq 2$ and  perform the following partial
Legendre transformation:
\begin{equation}\label{leg}
y_1=x_1, \quad y_i= u_i(x) \quad i\ge 2, \qquad u^*(y)=x' \cdot \nabla_{x'}u -u(x)
\end{equation}
with $x'=(x_2,...,x_n)$.
The function $u^*$ is obtained by taking the Legendre transform of
$u$ on each slice $x_1=const.$ We claim that $u^*$ (which is convex
in $y'$ and concave in $y_1$) satisfies
\begin{equation}\label{eqn-u*}
u^*_{11}+h(y_1)\, \det D^2_{y'}u^*=0.
\end{equation}
To see this we first notice that by the change of variable
$$v(x_1,x') \rightarrow u(x_1, x'+x_1\xi ')$$
$v$ satisfies the same equation as $u$ and
$$v^*(y)=u^*(y)- y_1 \, \xi' \cdot y'.$$
Thus we may assume that $D^2u$ is diagonal at $x$. Now it is easy
to check that
$$u^*_1=-u_1, \quad \nabla_{y'}u^*=x'$$
and
$$u^*_{11}=-u_{11}, \quad D^2_{y'}u^*=[D^2_{x'}u]^{-1}.$$
Hence $u^*$ satisfies \eqref{eqn-u*}.

\begin{rem}\label{rem1} The following hold:
\begin{enumerate}[$i$.]
\item The partial Legendre transform of $u^*$ is $u$, i.e. $(u^*)^*=u.$
\item The inequality $|u-v|\le \eps$ implies that $|u^*-v^*|\le \eps$ on their common domain of definition.
\item In dimension $n=2$,  the partial Legendre transform of the function $p(x_1,x_2)= a \, |x_1|^{2+\alpha} + b\, x_1x_2 + d\, x_2^2$
is given by
\begin{equation}\label{legp}
p^*(y_1,y_2)= (a \, |x_1|^{2+\alpha} + b\, x_1x_2 + d\, x_1^2)^*=-
a \, y_1^{2+\alpha} + \frac{1}{4d} (y_2-b\, y_1)^2.
\end{equation}
Notice that $p$ is a solution of the equation $\det D^2 u = c\,
|x_1|^\alpha$, for an appropriate constant $c$, and $p^*$ is a
solution of the equation $w_{11}+c\, |y_1|^\alpha \, w_{22}=0$.
\end{enumerate}
\end{rem}

\smallskip
We will restrict from now on our discussion  to   dimension  $n=2$ and the special case where $h(x_1)=|x_1|^\alpha$.

\begin{lem}{\label{lineq}}
Assume that for some $\alpha > 0$, $w$ solves the equation
$$Lw:=w_{11}+|y_1|^\alpha \, w_{22}=0 \qquad \mbox{in $B_1 \subset \R^2$}$$
with  $|w| \le 1$.
Then in $B_{1/2}$, $w$ satisfies
\begin{equation*}
\begin{split}
w(y)&=a_0+a_1\cdot y + a_2 \, y_1\, y_2 +\\
&+ a_3 \left(\frac{1}{2}\, y_2^2-\frac{1}{(\alpha+2)(\alpha+1)}\, |y_1|^{2+\alpha}\right ) + O((y_2^2+|y_1|^{2+\alpha})^{1+\delta})
\end{split}
\end{equation*}
with $|a_i|$ and $O(\cdot)$ bounded by a universal constant and
$\delta=\delta(\alpha)>0$.
\end{lem}

\begin{proof} First we prove that $w_2$ is bounded in the interior.
Since $Lw_2=0$, the same argument applied inductively  would  imply that the derivatives of $w$   with respect to  $y_2$ of any order  are bounded in the interior.

To establish the bound on $w_2$, we show that
\begin{equation}{\label{w2bd}}
L(C w^2+ \varphi^2w_2^2) \ge 0
\end{equation}
for a smooth cutoff function $\varphi$, to be made precise later. Indeed, a direct computation shows that
$$L(w^2)=2\, (w_1^2+|y_1|^\alpha w_2^2)$$
and
\begin{equation*}
\begin{split}
L(\varphi^2&w_2^2)=L(\varphi^2)w_2^2+\varphi^2L(w_2^2)+
2(\varphi^2)_1(w_2^2)_1 + 2|y_1|^\alpha(\varphi^2)_2(w_2^2)_2\\
&= L(\varphi^2)\, w_2^2+2\varphi^2\, (w_{21}^2+|y_1|^\alpha
w_{22}^2)+8(\varphi_1w_2)(\varphi
w_{21})+8|y_1|^\alpha(\varphi_2w_2)(\varphi w_{22})
\end{split}
\end{equation*}
hence
\begin{equation*}
\begin{split}
L(C w^2+ \varphi^2w_2^2) \geq 2C\, &|y_1|^\alpha w_2^2 +
2\varphi^2\, (w_{21}^2+|y_1|^\alpha w_{22}^2) \\&
+L(\varphi^2)w_2^2+8(\varphi_1w_2)(\varphi
w_{21})+8|y_1|^\alpha(\varphi_2w_2)(\varphi w_{22}).
\end{split}
\end{equation*}
By choosing the cutoff function  $\varphi$ such that $\varphi_1=0$
for $|y_1| \le 1/4$, then
$$L(\varphi^2) \geq
-C_1|y_1|^\alpha, \quad |\varphi_1 w_2|\le C_1|y_1|^{\alpha
/2}|w_2|$$ and we obtain (\ref{w2bd}) if $C$ is large. Therefore
$w_2$ is bounded in the interior by the maximum principle.

The equation $w_{11}+|y_1|^\alpha \, w_{22}=0$  and the bound  $|w_{22}| \leq C$  imply
the bound
$$|w_{11}| \le C\, |y_1|^\alpha.$$
Thus $w_1$ is bounded. The same estimates as above show that
 $w_{12}$, $w_{122}$ are bounded as well. By Taylor's formula, namely
$$f(t)=f(0)+f'(0)\, t+\int_0^t (t-s)\, f''(s)\, ds$$
and the equation $Lw=0$,  we conclude that
$$w(y_1,0)=w(0)+w_1(0)\, y_1-\frac{w_{22}(0)}{(\alpha+2)(\alpha+1)}\, y_1^{2+\alpha}+O(|y_1|^{3+\alpha}),$$

$$w(y_1,y_2)=w(y_1,0)+w_2(y_1,0)\, y_2+\frac{w_{22}(0)}{2}\, y_2^2+O(|y_2|^3+|y_1y_2^2|),$$
and
$$w_2(y_1,0)=w_2(0)+w_{12}(0)\, y_1+O(|y_1|^{2+\alpha})$$
from which  the lemma follows.

\end{proof}


%

\smallskip

\noindent{\em Notation:} By universal constants we understand positive constants that may also depend  on the exponent $\alpha$. Also,  when there is no
possibility of confusion we use the letters $c$, $C$
for various universal constants that change from line to line.

\section{Proof of Theorem \ref{t2}}{\label{sec1}}

Throughout this section we assume that $\alpha >0$, that $u$ satisfies
$$u(0)=0, \quad \nabla u (0)=0 $$
and we simply write $S_t$ for the section $S^u_{t,0}$.



Let
$$\Gamma:=\{\, |x_1|^{2+\alpha}+x_2^2 < 1\, \}$$
be the $1$ section of $|x_1|^{2+\alpha}+x_2^2$ at $0$.
If a set
$\Omega$ satisfies
$$ (1-\theta)\Gamma \subset \Omega \subset (1+\theta) \Gamma$$
we write
$$\Omega \in \Gamma \pm \theta.$$

The following approximation lemma constitutes the basic step in the proof of
Theorem \ref{t2}.

\begin{lem}{\label{l1}}
Assume that $u$ in the section $S_1$ satisfies
\begin{equation}\label{eqn-cf}
 \det D^2 u = c \, f(x), \qquad |f(x)- |x_1|^\alpha| \le \varepsilon
 \end{equation}
and
\begin{equation}\label{eqn-pm1}
S_1 \in \Gamma \pm \theta
\end{equation}
with $\eps \le \eps_0$ and  $ \eps ^ {1/8} \le \theta$, $\theta
<1$ small.  Then, for some small universal $t_0$, we have
$$S_{t_0} \in A \, D_{t_0}(\Gamma \pm \theta \, t_0^\delta)$$
where
$$A:=\begin{pmatrix}
  a_{11} & 0 \\
  a_{21} & a_{22}
\end{pmatrix} , \quad  D_{t_0}:=\begin{pmatrix}
  t_0^\frac{1}{2+\alpha} & 0 \\
  0 & t_0^{\frac{1}{2}}
\end{pmatrix}
$$
and
$$
|A-I| \le C \theta, \qquad \mbox{$C$ universal.}$$
Moreover, the constant $c$ in
\eqref{eqn-cf} satisfies \begin{equation}{\label{1.4}}
|c-2(1+\alpha)(2+\alpha)|\le C \theta.
\end{equation}

\end{lem}

\begin{proof}
We consider the solution
\begin{equation}{\label{1.1}}
v:=\frac{c^{1/2}}{[2(1+\alpha)(2+\alpha)]^{1/2}}(|x_1|^{2+\alpha}+x_2^2)
\end{equation}
of the equation
$$\det D^2 v =c \,  |x_1|^\alpha$$
and compute that
\begin{equation}{\label{1.2}}
\det D^2 (v+ \sqrt{c \eps}\,  |x|^2) > c \, (|x_1|^\alpha + \eps) \ge
\det D^2 u
\end{equation}
and
\begin{equation}{\label{1.2'}}
 \det D^2 (u + \sqrt{c \eps}\,  |x|^2) > c \, (f(x) + \eps) \ge \det D^2v
\end{equation}
because $|f(x)- |x_1|^\alpha| \le \varepsilon$,  by assumption.

We first notice that the assumption \eqref{eqn-pm1}  implies  that the constant
$c$ in equation \eqref{eqn-cf} is bounded from above by a universal constant, if $\eps_0$ is small. This can be easily seen from equation \eqref{1.2'} which,  with the aid of the
maximum principle,  implies that $u+ \sqrt{c \eps}\,  |x|^2 \geq v$, on $\{u=1\}$
(notice that both $v$ and $w=u+ \sqrt{c \eps}\,  |x|^2$ satisfy $v(0)=w(0)=0$
and $\nabla w(0)=\nabla v(0) =0$). Since $\{u=1\} \in \Gamma \pm \theta$,
this readily gives a bound on $c$,  if we  assume that $ \theta$ is  small.

We will next show that
\begin{equation}{\label{1.3}}
\{v<1\}\in \Gamma \pm 2 \theta
\end{equation}
which implies the bound \eqref{1.4}.
Indeed, if
$$\{v<1\} \subset (1- 2 \theta)\, \Gamma$$
\noindent then $v> u + \tilde c\,  \theta\,  |x|^2$ on $\{u=1\}$, for a universal $\tilde c$,
thus
$$v> u + \sqrt{c \eps}\, |x|^2, \qquad \mbox{on} \,\, \{u=1\}$$  since, by the assumptions of the lemma,  $\sqrt{\eps}  < \eps^{1/8} \leq\theta$ and
$\eps \leq \eps_0$, with $\eps_0$ sufficiently small.
We conclude
from the maximum principle (see (\ref{1.2})), that $v> u + \sqrt{c \eps}\,
|x|^2$ in $S_1$. This is a  contradiction,  since $u(0)=v(0)=0$.
If
\
$$(1+ 2 \theta)S_1 \subset \{v<1\}$$
\
then similarly we obtain $v+\sqrt{c \eps} |x|^2<u$ in $S_1$,
a contradiction.

Let $w$ be the solution of the problem
$$ \det D^2w=c\, x_1^2, \quad \mbox{in} \,\, S_1,\qquad w=u \quad \mbox{on $\partial S_1$}.$$
By the maximum principle
$$w +\sqrt{c \eps}\, (|x|^2-2 ) \le u \le w - \sqrt{c \eps}\, (|x|^2-2)$$
thus
$$|w-u| \le C \sqrt  \eps.$$
Also from (\ref{1.3}) we obtain
$$|w-v| \le C \theta.$$
Hence, by Remark \ref{rem1}, the corresponding partial Legendre
transforms defined in Section \ref{sec0} satisfy in $B_{1/2}$
\begin{equation}{\label{1.5}}
|w^*-v^*|\le C \theta
\end{equation}
\begin{equation}{\label{1.6}}
 |w^*-u^*|\le  C \sqrt \eps, \quad u^*(0)=0, \quad \nabla u^*(0)=0
\end{equation}
and $w^*$ and $v^*$ solve the same linear equation
$$w^*_{11}+c\, |y_1|^\alpha w^*_{22}=0.$$
Using Lemma \ref{lineq} for the difference $w^*-v^*$ together with
\eqref{legp}, (\ref{1.1}), (\ref{1.4}) and (\ref{1.5}), yields to
\begin{equation}\label{eqn-w1}
\begin{split}
w^*=-&|y_1|^{2+\alpha} +\frac{1}{4} y_2^2 + a+ b_1y_1+b_2y_2\\
& + \theta \, \left ( c\, y_1y_2 + d_1|y_1|^{2+\alpha} + d_2
y_2^2+O((|y_1|^{2+\alpha}+y_2^2)^{1+\delta}\, ) \right )
\end{split}
\end{equation} with the
coefficients $a,b_i,c,d_i$ bounded by a universal constant.

From (\ref{1.6}) we find that
$$w^*(0,y_2) \ge -C \sqrt \eps \quad \mbox{and} \quad  w^*(y_1,0) \le C \sqrt \eps$$
since, from the convexity in $y_2$ and concavity in $y_1$ of
$u^*$,
$$u^*(0,y_2) \geq 0 \quad \mbox{and} \quad u^*(y_1,0) \leq 0.$$
This and  \eqref{eqn-w1} imply  the bounds
$$|a| \le C  \eps^{1/2}, \quad |b_1| \le C \eps^{1/4}, \quad |b_2| \le C
\eps^{1/4}.$$
Thus, if $|y_1|^{2+\alpha}+ y_2^2 \le 10 \, t_0$,  then
$$w^*=-(1- d_1\theta)\, |y_1|^{2+\alpha} + \left (\frac{1}{4} + d_2 \theta \right )y_2^2+ c \theta \, y_1y_2 +
O(\eps^{1/4} + \theta \, t_0^{1+\delta}).$$ Hence, by performing the
partial Legendre transform on $w^*$ (using that $(w^*)^*=w$ and \eqref{legp}),  we obtain
\begin{equation}{\label{1.7}}
\left | w-[e_1|x_1|^{2+\alpha}+e_2\, (x_2+e_3\, x_1)^2] \right |  \le
C(\eps^{1/4} + \theta t_0^{1+\delta})
\end{equation}
for
$$|x_1|^{2+\alpha} + 4e_2^2\, (x_2+e_3x_1)^2 \le 10\, t_0$$
with $|e_1-1|,|e_2-1|, |e_3|$ bounded by $C \theta$.

\smallskip
We next observe that  if $p(x) = e_1|x_1|^{2+\alpha}+e_2\, (x_2+e_3\, x_1)^2$,
then the function
$$\tilde p(y) := \frac{1}{t_0} p(Fy)$$
with $F$ given by
$$F^{-1}:=\begin{pmatrix}
  t_0^{-\frac{1}{2+\alpha}} & 0 \\
  0 & t_0^{-\frac{1}{2}}
\end{pmatrix}
\begin{pmatrix}
  e_1^{\frac{1}{2+\alpha}} & 0 \\
  e_3 e_2^{\frac{1}{2}} & e_2^{\frac{1}{2}}
\end{pmatrix}=D_{t_0}^{-1}A^{-1}$$
satisfies
$$\tilde p(y) =
|y_1|^{2+\alpha}+y_2^2.$$
Hence, denoting by
$$\tilde{w}(y)=\frac{1}{t_0}w(Fy)$$
we conclude  from  (\ref{1.7})  that
$$|\tilde{w}(y)-(|y_1|^{2+\alpha}+y_2^2)| \le C(\eps^{1/4}t_0^{-1}
+ \theta t_0^{\delta}),  \qquad \mbox{for $|y_1|^{2+\alpha} + y_2^2
\le 2$}.
$$
\
Since $|\tilde{w}-\tilde{u}| \le C \eps^{1/2} t_0^{-1}$ (because $|w-u| \le C \eps^{1/2}$) we find
for  $\eps < \min(\theta^8,\eps_0) $, with $\eps_0 $ small,  that
$$\{ \tilde{u}<1 \} \in \Gamma \pm \gamma$$
with
$$\gamma=C(\eps^{1/4}\, t_0^{-1} + \theta \, t_0^{\delta}) \le
\theta \, t_0^{\delta'}.$$
The proof is now completed since
$S_{t_0} = F \{ \tilde {u}<1 \} = A\,  D_{t_0} \{ \tilde u < 1 \}.$

\end{proof}

\smallskip
The proof given above  also shows  the following Lemma.

\begin{lem}{\label{l2}} Assume that $u$  satisfies
$$ \det D^2 u = c \, f(x),  \qquad \mbox{on} \,\, S_1$$
and
$$ B_{1/k_0} \subset S_1 \subset B_{k_0}.$$
Then, given $\theta_0$, there exist $\eps_1(\theta_0,k_0)$ and
$t_1(\theta_0,k_0)$ small such that if
$$|f(x)- |x_1|^\alpha| \le \varepsilon_1$$
then
$$S_{t_1} \in A_0 D_{t_1}(\Gamma \pm \theta_0)$$
with
\begin{equation}{\label{A_0}}
A_0:=\begin{pmatrix}
  a_{0,11} & 0 \\
  a_{0,21} & a_{0,22}
\end{pmatrix}
\end{equation}
and
$$c(k_0)  \le a_{0,ii} \le C(k_0), \quad |a_{0,12}| \le C(k_0)$$
for some universal constants $c(k_0)$, $C(k_0)$.
\end{lem}

The proof of Theorem \ref{t2} readily follows from   the next
proposition which shows that if the section $S_\lambda$ has large
eccentricity,  for some $\lambda$,  then $u$ enjoys  the nonradial
behavior \eqref{nonrad} at the origin.

\begin{prop}{\label{p1}}
Assume that $u$  solves  the equation
$$\det D^2 u = |x|^\alpha, \qquad \mbox{on} \,\, S_1$$
and that $S_1$ has large eccentricity, i.e.
$$ F B_{1/k_0} \subset S_1 \subset F B_{k_0}, \quad F:=c\begin{pmatrix}
  b & 0 \\
  0 & 1/b
\end{pmatrix} $$
with $b \ge C_0$. Then,  there exists a $z$-system of coordinates
such that
\
\begin{equation}{\label{1.8}}
u(z)=\frac{a}{(\alpha+2)(\alpha+1)}|z_1|^{2+\alpha}+\frac{1}{2a}z_2^2
+O \left ( (|z_1|^{2+\alpha}+z_2^2)^{1+\delta} \right)
\end{equation}
for some $a>0$.

\end{prop}

\begin{proof} The proof  will be based on an inductive argument, where at each step will
use Lemma \ref{l1}.

Denote by
$$ v_1 (x):= u(Fx),$$
and compute that $v_1$ satisfies the equation
$$ \det D^2 v_1 (x)= (\det F )^2 |Fx| ^\alpha = c^{4+\alpha}  b^\alpha |(x_1, b^{-2}x_2)|^\alpha.$$
Also,
$$ \{v_1 < 1 \}=F^{-1}S_1.$$
If $b$ is large, then $v_1$ satisfies  hypothesis of the Lemma
\ref{l2}. Hence,  for some fixed $\theta_0$ we obtain
$$ S_{t_1} = F \{ v_1 < t_1 \} \in  F A_0 D_{t_1}(\Gamma
\pm \theta_0)$$ with $A_0$ satisfying (\ref{A_0}).

We assume by induction that for $t=t_1 t_0^k$ we have

$$ S_t \in F A_k D_t (\Gamma \pm \theta_0 t_0^{(k-1)
\delta })$$ with
$$A_k:=\begin{pmatrix}
  a_{k,11} & 0 \\
  a_{k,21} & a_{k,22}
\end{pmatrix} $$
and
\begin{equation}{\label{a_bd}}
 \quad  c/2 \le a_{k,ii} \le 2C, \quad |a_{k,21}| \le 2C.
\end{equation}
We will  show that

$$ S_{t_0t} \in F A_{k+1} D_{t_0 t} (\Gamma \pm \theta_0
t_0^{k\delta})$$
where
$$A_{k+1}=A_k\, E_k$$
and
$$E_k:=\begin{pmatrix}
  e_{k,11} & 0 \\
  e_{k,21} & e_{k,22}
\end{pmatrix} $$
with
\begin{equation}{\label{e_bd}}
|e_{k,ii}-1| \le C\theta_0 t_0^{(k-1) \delta}, \quad
|e_{k,21}|\, t^{-\frac{\alpha}{2(2+\alpha)}} \le C\theta_0 t_0^{(k-1)
\delta}.
\end{equation}\\
Notice that condition \eqref{e_bd} implies the bound
\begin{equation}{\label{a-a}}
|A_{k+1}-A_k| \le C \theta_0 t_0^{(k-1) \delta}.
\end{equation}

To prove this inductive step, we observe that the  function
$$  v_t (x) := t^{-1} u (F A_k D_t
x)$$ satisfies in $ \{ v_t <1 \} $ the equation

$$\det D^2 v_t = c_t\,  |\tilde x|^\alpha$$
with
$$|\tilde x|^\alpha=\left | \left (a_{k,11}t^\frac{1}{2+\alpha}x_1,b^{-2}
(a_{k,21}t^\frac{1}{2+\alpha}x_1+a_{k,22}t^\frac{1}{2}x_2  )
 \right ) \right |^\alpha=c_t'\, f_t(x)$$
and
$$|f_t-|x_1|^\alpha|\le b^{-2}t^\frac{\alpha}{2(2+\alpha)}.$$
Also
$$\{ v_t <1 \}\in \Gamma\pm \theta_0 t_0^{(k-1)\delta}$$

\medskip
\noindent since $ S_t \in F A_k D_t (\Gamma \pm \theta_0 t_0^{(k-1)
\delta })$ by the inductive assumption.
Hence, if  $\delta=\delta(\alpha)$  is chosen  small, then $v_t$ satisfies the
assumptions of Lemma \ref{l1}, yielding to

$$\{ v_t <t_0 \} \in \tilde A D_{t_0} (\Gamma \pm \theta_0
t_0^{k\delta})$$ with
\begin{equation}{\label{1.9}}
|\tilde A - I| \le C \theta_0 t_0^{(k-1)\delta}.
\end{equation}
Thus
$$S_{t_0t} \in F A_k D_t \tilde A D_{t_0}(\Gamma \pm
\theta_0 t_0^{k\delta}).$$
\
Defining $E_k$ such that
$$D_t \tilde A = E_{k} D_t,$$
we see from (\ref{1.9}) that $E_k$ satisfies (\ref{e_bd}). We
conclude the proof of
 the induction step by first choosing $\theta_0$ small so
that (\ref{A_0}) and (\ref{a-a}) imply that (\ref{a_bd}) is always
satisfied.

Denote by
$$A^*:= \lim_{k \to \infty}A_k.$$
We will  prove next that
\begin{equation}{\label{1.10}}
S_t \in F A^* D_t (\Gamma \pm C' t^{\delta}).
\end{equation}

\noindent As before, let $t=t_1 t_0^k$. Notice that
$$ A^*=A_k E_k^*, \quad
E_k^*:= \Pi _{i=k}^{\infty}E_i$$ and it is straightforward to
check from (\ref{e_bd}) that

\begin{equation}{\label{e*_bd}}
|e^*_{k,ii}-1| \le C_1 t^{\delta}, \quad
|e^*_{k,12}|\, t^{-\frac{\alpha}{2(2+\alpha)}} \le C_1 t^\delta.
\end{equation}

\noindent We have
$$ A_k D_t = A^*  (E_k^*)^{-1} D_t = A^* D_t
\tilde E$$
 with
$$|\tilde e_{k,ii}-1| \le C_2 t^{\delta},
\quad |\tilde e_{k,12}| \le C_2 t^{\delta}.$$
\noindent Now (\ref{1.10}) follows since
$$\tilde E (\Gamma \pm C_2 t^\delta) \subset \Gamma \pm
C' t^\delta.$$

\medskip
Finally, from (\ref{1.10}) we see that

$$u(FA^*x) = |x_1|^{2+\alpha} + x_2^2 + O((|x_1|^{2+\alpha} + x_2^2)^{1 + \delta})$$

\noindent which implies that in a $z$-system of coordinates

$$u(z)=\beta_1|z_1|^{2+\alpha} + \beta_2\, z_2^2 + O((|z_1|^{2+\alpha} + z_2^2)^{1 + \delta}).$$

\noindent The rescaled functions
$$r^{-1}u(r^\frac{1}{2+\alpha}z_1,r^\frac{1}{2}z_2)$$
converge,  as $r \to 0$,  to
$$\tilde u(z):=\beta_1|z_1|^{2+\alpha} + \beta_2z_2^2.$$
Moreover, this function solves the limiting equation
$$\det D^2 \tilde u=|z_1|^\alpha.$$

\noindent Hence
$$2(2+\alpha)(1+\alpha)\beta_1\beta_2=1$$
which implies \eqref{1.8}.
\end{proof}

\

\section{Negative powers}{\label{sec3}}

In this section we consider the equation
\begin{equation}{\label{3.eq}}
\det D^2 u=|x|^\alpha,\qquad \mbox{in} \ \Omega \subset \R^2
\end{equation}
in the negative range of exponents $-2 < \alpha < 0$.

\noindent  We will assume, throughout the section,   that $0 \in \Omega$ and
$$u(0)=0, \,\, \nabla u(0)=0.$$

\noindent Our goal is to prove the  following proposition, which shows  that solutions of equation \eqref{3.eq} admit only the radial behavior near the origin. This is in contrast
with the case $0 < \alpha < \infty$, where both the radial behavior  and
the non-radial behavior \eqref{1.8} occur (see Proposition \ref{p1}).

\begin{prop}{\label{p3.1}}
There exist positive constants $c$, $C$ (depending on $u$) such that
$$c\, |x|^{2+\alpha /2}\le u(x) \le C \, |x|^{2+\alpha /2}$$
near the origin.
\end{prop}

We distinguish two cases depending on whether or not the measure
$|x|^{\alpha}dx$ is doubling with respect to all ellipsoids (see
the discussion in Section \ref{sec0}).

\

\noindent{\bf i. The case $-1<\alpha<0:$} In this case the measure
$$\mu:=|x|^\alpha dx$$ is doubling with respect to ellipsoids. Indeed, it suffices to show that  there exists $c>0$ such that for any ellipsoid $E$, we have
\begin{equation}{\label{3.1}}
 \mu(x_0+E)\ge c \, \mu(x_0+2E).
 \end{equation}
Since $-1<\alpha$, the density
$$(x_1^2+x_2^2)^{\alpha /2}$$
is doubling on each line $x_2=const.$ with the doubling constant
independent of $x_2$. This implies that the density $\mu=|x|^\alpha\, dx$ is
doubling with respect to any line in the plane. From this and the
fact that $x_0+2E$ can be covered with translates of $x_0+E/2$
over a finite number of directions we obtain (\ref{3.1}).

From Theorem \ref{caf1}, there exists a matrix $A_t$ such that
$S_t \sim A_t$, i.e

\begin{equation}{\label{3.2}}
k_0^{-1}A_t B_r \subset S_t \subset k_0A_t B_r,
\end{equation}
with
$$r=t\,  (\mu(S_t))^{-1/2}, \quad \quad \det
A_t=1.$$
\smallskip

In this case Proposition \ref{p3.1} follows from the lemma below.

\begin{lem}{\label{l3.1}} There exist universal constants $C >0$ large and $\delta >0$,
such that if $S_t \sim A_t$ with $|A_t| >C$,  then
\begin{equation}\label{eqn-Sd}
S_{\delta t} \sim A_{\delta t}, \qquad
\mbox{ with\,  $|A_{\delta t} |\le |A_t|/2$} .
\end{equation}
In particular,  $|A_t| \le C |A_{t_0}|$,  if $t \le t_0$.

\end{lem}

\begin{proof} We will use a compactness argument.
Assume, by contradiction,  that the conclusion of the lemma is not
true.   Then we can find a sequence of solutions  $u_k$ of
\eqref{3.eq} with sections $S^{u_k}_{t_k}$ at $0$ such that
$S_{t_k}^{u_k} \sim A_{t_k}^{u_k}$ with $|A^{u_k}_{t_k}|\to
\infty$ and \eqref{eqn-Sd} does not hold for any $\delta >0$.

Without loss of generality we may assume that

\begin{equation}{\label{3.3}}
A^{u_k}_{t_k}:=\begin{pmatrix}
  a_k & 0 \\
  0 & a_k^{-1}
\end{pmatrix}, \quad a_k \to \infty.
\end{equation}

\noindent We renormalize the functions $u_k$ as

\begin{equation}\label{eqn-ren}
v_k(x):=\frac{1}{t_k}u_k(r_kA^{u_k}_{t_k}x)\end{equation}
so that
 $$\det D^2
v_k=c_k\, |A^{u_k}_{t_k}x|^\alpha=c_k'\, |x_1^2+a_k^{-4}x_2^2|^{\alpha
/2}$$
and
$$ k_0^{-1}B_1 \subset S^{v_k}_1 \subset k_0B_1.$$

\smallskip
\noindent Since,  $S_{t_k}^{u_k} \sim A_{t_k}^{u_k}$, i.e  in particular  $r_k = t_k \, (\mu(S_{t_k}^{u_k}))^{-1/2}$, the Monge Amp\'ere measure $\det D^2 v_k \, dx$
satisfies

$$ \det D^2 v_k(S^{v_k}_1)=r_k^2t_k^{-2}\mu
(S^{u_k}_{t_k})=1.$$

\noindent
Hence, as  $k \to \infty$ we can find a subsequence of the $v_k$'s that
 converge uniformly to a function $v$ that satisfies

\begin{equation}{\label{3.4}}
\det D^2 v=c\, |x_1|^\alpha dx
\end{equation}
and
$$k_0^{-1}B_1 \subset S^v_1 \subset k_0B_1,
 \quad \det D^2 v_k(S^v_1)= 1.$$

\medskip
\noindent Obviously, the constant $c$ in (\ref{3.4}) is bounded from above and
below by universal constants. Since the right hand side of
(\ref{3.4}) does not depend on $x_2$ and $v$ is constant on
$\partial S^v_1$, Pogorelov's interior estimate holds and we
obtain the bound

$$v_{22} < C_1, \qquad \mbox {in $(2k_0)^{-1}B_1$}.$$

This implies that the section $S^v_\delta $  contains a segment
of size $\delta^{1/2}$ in the $x_2$ direction, namely

\begin{equation}{\label{3.5}}
\{x_1=0,\,  |x_2|\le (\delta/C_1)^{1/2} \} \subset S^v_\delta.
\end{equation}
From  Theorem \ref{caf1}  there exists

\begin{equation}{\label{3.6}}
A_\delta=\begin{pmatrix}
  a & 0 \\
  b & a^{-1}
\end{pmatrix}, \quad 0<a<C(\delta), \quad |b|\le C(\delta)
\end{equation}
with
\begin{equation}{\label{3.7}}
k_0^{-1}A_\delta B_r \subset S^v_\delta \subset k_0A_\delta
B_r
\end{equation}
and
\begin{equation}{\label{3.8}}
r=\delta\,  [\det D^2 v(S^v_\delta)]^{-1/2}.
\end{equation}

\noindent  From (\ref{3.5}) and \eqref{3.7} we have

$$\frac{r}{a} \ge c_1\,  \delta^{1/2}$$
while from (\ref{3.6}), (\ref{3.7}) and  (\ref{3.8}) we get

$$\delta^2=r^2 \, \det D^2 v(S_\delta^v)\ge c_2 \, r^2 \, \frac{r}{a}(ar)^{1+\alpha}.$$

\noindent From the last two inequalities we obtain

\begin{equation}{\label{3.9}}
a \le C_2\delta^\frac {-\alpha}{4(2+\alpha)}\le 1/4 \quad \mbox{for
$\delta$ small universal}.
\end{equation}

\medskip

Since the $v_k$'s converge uniformly to $v$,   their $\delta$
sections also converge uniformly, thus
$$S^{v_k}_\delta \sim A_\delta,  \qquad \mbox{for $k$ large}$$
and hence
$$S^{u_k}_{\delta t_k} \sim A^{u_k}_ {t_k} A_\delta.$$
From (\ref{3.3}), (\ref{3.6}), (\ref{3.9}) we conclude
$$|A^{u_k} _{t_k} A_\delta| \le |A^{u_k}_ {t_k}|/3
 \quad \mbox{for $k$ large,}$$
which implies that the function $u_k$ satisfies \eqref{eqn-Sd} , a contradiction.

\end{proof}

\medskip

\noindent {\bf ii. The case $-2 < \alpha \le -1$}:   \ In this case the measure $\mu$ is not doubling with
respect to any convex set but it is still doubling with respect to
convex sets that have the origin as the center of mass.

We proceed as in the first case but replacing the sections $S_t$
with the sections $T_t$ that have $0$ as the center of mass. The
existence of these  sections follows from  the following lemma due
to L. Caffarelli, Lemma 2 in \cite{Ca2}.

\begin{lem} [Centered sections] {\label{bs}} Let $u:\mathbb{R}^n \to \mathbb{R} \cup \{\infty\}$ be a globally defined
convex function (we set $u=\infty$ outside $\Omega$). Also, assume
$u$ is bounded in a neighborhood of $0$ and the graph of $u$ does
not contain an entire line.

Then, for each $t>0$, there exists a ``$t-$ section"
$T_t$ centered at $0$,
that is there exists $p_t$ such that the convex set
$$T_t:=\{ \, u(x) < u(0)+ p_t \cdot x + t \, \}$$
is bounded and has $0$ as center of mass.
\end{lem}

Using the lemma above one can obtain   Theorem \ref{caf1}
(similarly as in \cite{Ca1}), with  $S_t$ is
 replaced by $T_t$:   {\it for every $T_t \subset
\Omega$ as above, there exists a unitary matrix $A_t$, such that

\begin{equation}{\label{3.10}}
k_0^{-1}A_t B_r \subset T_t \subset k_0A_t B_r
\end{equation}
with
$$r=t \, (\mu(T_t))^{-1/2}.$$
If $(\ref{3.10})$ is satisfied we write $T_t \sim A_t.$}

\medskip
We will next show the analogue  of Lemma \ref{l3.1} for this case.

\begin{lem}{\label{l3.2}} There exist  universal constants $C >0$ large and $\delta >0$,
such that if $T_t \sim A_t$ with $|A_t| >C$,  then     $ T_{\delta t} \subset T_t$ and
\begin{equation}\label{eqn-Td}
 T_{\delta t} \sim A_{\delta t}, \qquad \mbox{ with $|A_{\delta t}|\le |A_t|/2$}.
 \end{equation}
\end{lem}

\begin{proof} We argue similarly as in the proof of Lemma \ref{l3.1}. We assume by
contradiction that the conclusion does not hold for a sequence of
functions $u_k$. Proceeding as in the
proof of lemma \ref{l3.1}, we work with the renormalizations $v_k$ of
$u_k$ defined by \eqref{eqn-ren}  which satisfy

$$\det D^2
v_k=c_k'\, |x_1^2+a_k^{-4}x_2^2|^{\alpha /2}=:\mu_k$$
and
$$ k_0^{-1}B_1 \subset T^{v_k}_1 \subset k_0B_1, \quad \det
D^2v_k(T^{v_k}_1)=1. $$

\noindent As $k \to \infty$,  we can find a subsequence of the $v_k$'s  which
 converges  uniformly to a function $v$. Since $a_k \to \infty$ and $-2<
\alpha \le -1$, the corresponding measures $\mu_k$, when
restricted to a line $x_2=const.$,  converge weakly to the measure
$c \,  |x_2|^{1+\alpha}\delta_{\{x_1=0\}}$. This implies that the
measures $\mu_k$ converge weakly to $c\,
|x_2|^{1+\alpha}d\mathcal{H}^1_{\{x_1=0\}},$ where $d\mathcal{H}^1$
is the 1 dimensional Hausdorff measure.
Hence,  the limit function $v$ satisfies

\begin{equation}{\label{3.11}}
\det D^2 v=c\, |x_2|^{1+\alpha}\, d\mathcal{H}^1_{\{x_1=0\}}
\end{equation}

$$k_0^{-1}B_1 \subset T^v_1 \subset k_0B_1, \qquad \det D^2 v_k(T^v_1)=
1.$$
Clearly $c$ is bounded from above and below by universal constants.

We notice that the measure $d\mathcal{H}^1_{\{x_1=0\}}$ is doubling with
respect to any convex set with the center of mass on the line $\{x_1=0\}$.
Using the same methods as in the case of classical Monge-Amp\'ere equation
one can show that the graph of $v$ contains no line segments when
restricted to
$\{x_1=0\}$ (see the Lemma \ref{l3.3} below). From this and the fact that $v$ is the
convex envelope of its restriction on $\partial T_1^v$ and $\{x_1=0\}$
(see (\ref{3.11})) we conclude that there exist two supporting
planes with slopes $\beta e_2 \pm \gamma e_1$ to the graph of $v$ at $0$. Moreover,
it follows from the compactness of the equation (\ref{3.11})  that $\gamma$ can be
chosen universal, and the sections $T^v_\delta$ satisfy

$$T^v_\delta \subset (2k_0)^{-1}B_1$$
when $\delta \leq \delta_0$,  a universal constant. We have

\begin{equation}{\label{3.12}}
T^v_\delta \subset \{|x_1| \le c(\gamma)\delta \}.
\end{equation}
Let $A_\delta$ be of the form (\ref{3.6}) with

\begin{equation}{\nonumber}
k_0^{-1}A_\delta B_r \subset T^v_\delta \subset k_0A_\delta
B_r
\end{equation}
and
\begin{equation}{\label{3.13}}
r=\delta\,  [\det D^2 v(T^v_\delta)]^{-1/2} \sim \delta
(r/a)^{1+\alpha/2}.
\end{equation}

\medskip
\noindent
On the other hand (\ref{3.12}) implies

$$a\,  r \le c \delta$$
which together with (\ref{3.13}) yields

$$a\le c \,  \delta^{\frac{2+\alpha}{6+\alpha}} \le 1/4$$
for $\delta$ small enough. Now the contradiction follows as in
Lemma \ref{l3.1}.
\end{proof}

\begin{lem}{\label{l3.3}}    If $v$ satisfies $(\ref{3.11})$, then
$$v_0(t):=v(0,t)$$ is strictly convex.

\end{lem}

\begin{proof}
 Assume that the conclusion does not hold. Then,
after subtracting a linear function, we can assume that

$$v \ge 0 \qquad \mbox{in $T^v_1$}$$
and
$$\quad v_0(t)=0 \quad \mbox{for $t\le 0$}, \quad \quad v_0(t)>0 \quad
\mbox{for $t> 0$}.$$
Let $$l_\eps:=\eps t + a_\eps $$ be such that

\begin{equation}{\label{3.14}}
0 \in \{v_0 <
l_\eps\}=(b_\eps, c_\eps) \rightarrow 0, \qquad \frac{c_\eps}{| b_\eps|}
\to 0 \quad \mbox{as $ \eps \to 0$}.
\end{equation}

\noindent
We consider the linear function $p_\eps$ in $\mathbb{R}^2$ such that $\{u
< p_\eps\}$ has center of mass on $\{x_1 = 0\}$ and $p_\eps=l_\eps$ on
$\{x_1 = 0\}$.

We claim that for $\eps$ small, $\{u< p_\eps\}$ is compactly included in
$T_1^v$. Otherwise, the graph of $v$ would contain a segment passing
through $0$, hence $v = 0$ in an open set which intersects the line $\{x_1
= 0\}$ and we contradict (\ref{3.11}).

Since $d\mathcal{H}^1_{\{x_1=0\}}$ is doubling with
respect to the center of mass of $\{u< p_\eps\}$, we conclude that this
set is also balanced around $0$ which contradicts (\ref{3.14}).
\end{proof}

\medskip

We are now in the position to exhibit the final steps of the proof
of Proposition \ref{p3.1} in the case $-2 < \alpha \leq -1$.

\medskip

\noindent {\it Proof of Proposition \ref{p3.1}:}
We choose $t_0$ small, such that $$T_{t_0} \subset \Omega.$$
The existence of $t_0$ follows from the fact that the graph of $u$
cannot contain any line segments.

From Lemma \ref{l3.2} we conclude that there exists a large
constant $K
> 0$ depending on the eccentricity of $T_{t_0}$ such
that

$$ T_t \sim A_t, \qquad \mbox{with $|A_t|\le K$
for all $t \le \delta t_0$}.$$

\

\noindent {\it Claim:} There exists $\gamma$ depending on $K$ such that $ S_{\gamma
t} \subset T_t.$

\medskip
\noindent  To show this, first observe that by  rescaling we can
assume that $t =1$. We use the compactness of the problem for
fixed $K$. If there exist a sequence $\gamma_k \to 0$ and
functions $u_k$ for which  the conclusion does not hold then, the
graph of the limiting function $u_\infty$ (of a subsequence of
$\{u_k\}$)  contains a line segment. This is a contradiction since
$u_\infty$ solves the Monge-Amp\'ere equation (\ref{3.eq}), which
proves the claim.

\noindent If $t=1$, then from simple geometrical considerations and the claim
above we obtain
$$\gamma k_0^{-1}K^{-1} B_1 \subset S_{\gamma} \subset k_0 K B_1.$$

\noindent By rescaling, we find that $S_t$ has bounded
eccentricity for $t$ small, and the proposition is proved.

\qed

\

\section{Homogenous solutions and blowup limits}{\label{sec4}}

We will consider in this section  homogenous solutions of the equation

$$\det D^2 w (x)=|x|^\alpha \qquad \mbox {in
$\mathbb{R}^2$}$$
for $ \alpha > -2$, namely solutions of the form

$$w(x)=r^{2+ \alpha /2}g (\theta):=r^\beta g, \qquad \beta = 2+ \alpha /2.$$
In the polar system of coordinates

$$ D^2 w (x)=r^{\beta-2}
\begin{pmatrix}
  \beta (\beta-1)g & (\beta-1)g' \\
  (\beta-1)g' & g''+\beta g
\end{pmatrix}.$$

\noindent Thus, the function $g$ satisfies the following ODE

\begin{equation}{\label{4.1}}
 \beta g (g''+ \beta g) - (\beta -1) (g')^2=1/(\beta-1).
\end{equation}

\noindent We consider $g$ as the new variable in a maximal interval $[a,b]$ where
$g$ is increasing,  and define $h$ on $[g(a),g(b)]$ as

$$g'=\sqrt{2h(g)}.$$
We have $$g''=h'(g)$$ thus $h$ satisfies

$$\beta t \, (h'(t) + \beta t)- 2(\beta -1)\, h(t)=1/(\beta -1).$$

\noindent Solving for $h$ we obtain
\begin{equation}{\label{h}}
2\, h_c(t)=c\, t^{2(1-\frac{1}{\beta})}-\beta^2t^2-\frac{1}{(\beta
-1)^2}
\end{equation}
for some $c$ positive.

\noindent The function $g$ on $[a,b]$ is the inverse of

$$a + \int_{g(a)}^\xi \frac{1}{\sqrt{2h_c(t)}}dt$$
and the length of the interval $[a,b]$ is given by

\begin{equation}{\label{4.2}}
b-a=\int_{\{h_c>0\}} \frac{1}{\sqrt{2h_c(t)}}dt:=I_c.
\end{equation}

\noindent Solutions of (\ref{4.1}) are periodic, of period $2(b-a)$, thus
a global solution $g$ on the circle exists if and only if $I_c$ equals
$\pi / k$,  for some integer $k$. Next we investigate the existence of such
solutions.

First we notice that for any quadratic polynomial $f(s)=-l^2\, s^2 + d_1\, s+d_2$ of opening $-2\, l^2$,  we have

\begin{equation}{\label{4.3}}
\int_{\{f>0\}} \frac{1}{\sqrt{f(s)}}ds=\frac{\pi}{l}.
\end{equation}

\medskip
\noindent
Therefore if $\phi(s)$ denotes  any  convex function which intersects the parabola
$l^2 s^2$ at two points, and we set $f(s) =- l^2 s^2+d_1 s+d_2$, with
 $d_1 s+d_2$ denoting  the
line through  the intersection points between $\phi(s)$ and $ l^2 s^2$, then

\begin{equation*}
\int_{\{ \phi(s) - l^2 s^2  >0\}} \frac{1}{\sqrt{\phi(s) - l^2 s^2
}}\, ds \geq \int_{\{f>0\}} \frac{1}{\sqrt{f(s) }}\, ds =
\frac{\pi}{l}.
\end{equation*}

\medskip

\noindent If $\phi(s)$ is concave we obtain the opposite inequality.

\medskip
\noindent  Applying the above to $h_c(s)$, we find that  depending on the convexity of the first term in (\ref{h}), we
obtain that the integral $I_c$ in (\ref{4.2}) is less (or greater) than $\pi / \beta$
for   $\beta<2$ (or $\beta>2$), i.e.,
\begin{equation}\label{eqn-beta1}
I_c <  \frac \pi \beta, \quad \mbox{if}\,\,  \beta <2 \qquad \mbox{and} \qquad I_c >  \frac \pi \beta, \quad \mbox{if} \,\, \beta >2.
\end{equation}

\noindent On the other hand, by performing the change of variable
$$t=s^\frac{\beta}{2}$$
in the integral (\ref{4.2}) we obtain the integral (\ref{4.3}) with

$$f(s):= c_1 s-4s^2-c_2s^{2-\beta}$$

\noindent for some positive constants $c_1$, $c_2$ depending on $c$. Hence,
depending  on
the convexity of the last term of $f$, the integral $I_c$ is greater (or less) than  $\pi / 2$
for  $\beta <2$ (or $\beta>2$), i.e.,
\begin{equation}\label{eqn-beta2}
I_c >   \frac \pi 2, \quad \mbox{if}\,\,  \beta <2 \qquad \mbox{and} \qquad I_c <  \frac \pi 2 \quad \mbox{if} \,\, \beta >2.
\end{equation}

\medskip
Let $- 2 < \alpha < 0$, or equivalently  $1< \beta < 2$.   It follows from \eqref{eqn-beta1} and \eqref{eqn-beta2} that $\pi/2 < I_c < \pi/\beta$, hence $I_c = \pi/k$, for
an integral $k$ only when $k=1$. This readily implies that the only homogeneous solution
in this case is the radial one.

\medskip
Assume next that $\alpha >0$. We will show next that in this case, depending on the value
of $\beta$,  more homogeneous solutions may exist.

\noindent  To this end, denote by $c_0=c_0(\alpha)$ the value of $c$ for which the two
functions

$$f_1(t)=c\, t^{2(1-\frac{1}{\beta})} \quad
\mbox {and} \quad  f_2(t)= \beta^2 t^2 + \frac{1}{(\beta-1)^2}$$

\noindent become tangent. When $c < c_0$, then the set were $h_c(t) >0$
is empty. As $c \to c_0^+$ the set $\{ t: \, h_c(t) >0\}$ approaches the point
$t_0$ at which the two functions $f_1(t)$ and $f_2(t)$  become tangent when $c=c_0$. Since $f'_1(t_0)= f'_2(t_0)$ when $c=c_0$, the point $t_0$ satisfies
$$2  c\, (1-\frac 1\beta) \, t_0^{1-\frac 2\beta} = 2\beta^2 t_0$$
which implies that
$$c\, (1-\frac 1\beta) \, t_0^{-\frac 2\beta} =\beta^2.$$

\noindent As  $c \to c_0^+$, $f_1(t)$ behaves as its Taylor  quadratic
polynomial, namely
$$f_1(t) \approx f(t_0) + f'(t_0)\, t^2 + \frac{f''(t_0)}{2}\, t^2$$
and
$$\frac{f''(t_0)}2=  c \, (1-\frac 1\beta) \, (1-\frac 2\beta) \, t_0^{-\frac 2\beta} =
\beta^2 (1-\frac 2\beta).$$

\noindent We conclude that,  as $c \to c_0^+$, $(h_c)^+$ behaves as a quadratic  polynomial
of opening $-4\beta$, and thus $I_c$ converges to
$\pi/\sqrt{2 \beta}$. Hence, $(\frac{\pi}{\sqrt{2\beta}},\frac{\pi}{2}) \subset \{I_c, \,\,  c >c_0
\}$ and also $ \{I_c, \,\, c >c_0
\} \subset (\frac{\pi}{\beta},\frac{\pi}{2})$, by \eqref{eqn-beta1}, \eqref{eqn-beta2}.

\medskip
Summarizing  the discussion above yields:

\begin{prop}{\label{p4.1}}
Homogenous solutions to $(\ref{eq})$ are periodic on the unit
circle.

\noindent i.  If $-2<\alpha<0$,  then the only homogenous solution is the radial one.

\noindent ii.  If $\alpha>0$, then there
exists a homogenous solution of principal period $2\pi /k$ if and only if
 $$ \frac{\pi}{k} \in \{I_c, \,\,  c>c_0(\alpha) \}.$$
In addition,
$$ (\frac{\pi}{\sqrt{2\beta}},\frac{\pi}{2}) \subset \{I_c, \,\, c >c_0
\}
\subset (\frac{\pi}{\beta},\frac{\pi}{2})$$
with $\beta = 2 + \alpha/2$.
\end{prop}

\medskip

Using the proposition above, we will now prove Theorem \ref{t4}. We begin
with two useful remarks.

\begin{rem}
From (\ref{h}) we see that any point in the positive quadrant can
be written as $(t,\sqrt{2h_c})$ for a suitable $c$. Hence, given
any point $x_0\in
\partial B_1$ and any positive
symmetric unimodular  matrix
$A$, there exists a homogenous solution $w$ in a neighborhood of
$x_0$ such that $D^2w(x_0)=A$.
\end{rem}

\begin{rem} Equation  (\ref{h}) gives
$$[(h'+\beta t) + \beta
(\beta-1)t]\, t^{\frac{2}{\beta}-1}=c(1-\frac{1}{\beta})$$
hence
$$\Delta w \, [r^2 w_{rr}]^{\frac{2}{\beta}-1}$$
is constant for any local homogenous solution $w$.  This
quantity  will play a crucial role in the proof of Theorem \ref{t4}.
\end{rem}

\begin{defn}
For any solution $u$ of equation \eqref{eq}, we define
$$J_u(x):= (\Delta  u) (r^2 u_{rr})^\gamma, \qquad
\gamma:=\frac{2}{\beta}-1.$$
\end{defn}

\begin{rem}
The  quantity $J_u(x)$ remains invariant under the homogenous scaling
$$v(x)=r^{-\beta}u(r x), \quad J_v(x)=J_u(r x).$$
We denote by $J_0$ the constant obtained when we evaluate $J$ on the
radial solution $u_0$ of (\ref{eq}).
\end{rem}

\begin{prop}{\label{p4.2}} The function
$$ |J_u-J_0|$$
cannot have an interior maximum in $\Omega \setminus \{0\}$ unless it is
constant.
\end{prop}

\begin{proof} We compute the linearized operator $u^{ij}M_{ij}$ for
$$M=
\log J_u= \log (\Delta u) + \gamma \log (x_i x_j u_{ij})$$
at a point
$x \in \Omega \setminus \{0\}$ where $J_u(x) \ne J_0$.

By choosing an appropriate system of coordinates and by rescaling,
we can assume that $|x|=1$ and $D^2u$ is diagonal. By differentiating the
equation \eqref{eq} twice we obtain
\begin{equation}{\label{4.4}}
u^{ii}u_{kii}= \alpha x_k
\end{equation}
and
\begin{equation}{\nonumber}
u^{ii}u_{klii}= u^{ii}u^{jj}u_{kij}u_{lij} + \alpha (\delta^l_k -2
x_kx_l).
\end{equation}

Since the linearized equation of each second derivative of $u$ depends
on $D^3u$, $D^2u$ and  $x$ we see that
\begin{equation}\label{eqn-mij}
u^{ij}M_{ij}=H(D^3u,D^2u,x)
\end{equation}
where $H$ is a quadratic polynomial in $D^3u$ for fixed $D^2u>0$ and
$x$.

Let $w$ denote the (local) homogenous solution for which
$D^2u(x)=D^2w(x)$.
Since $M_w=\log J_w$ is constant, we have
$$H(D^3w,D^2w, \cdot )=0$$
in a neighborhood of $x$.

\medskip

\noindent{\em Claim.} We have
$$\|D^3u(x)-D^3w(x)\| \le C|\nabla M|$$
with the constant $C$ depending on $D^2 u$ and $x$.

\medskip

\noindent {\em Proof of Claim.} \, From (\ref{4.4}) and the following equalities

$$ M_k= \frac{u_{iik}}{\Delta
u}+ \gamma \frac{x_ix_ju_{ijk}+2x_iu_{ik}}{x_ix_ju_{ij}}$$
we obtain the following system for the third
derivatives of $u$,

$$\begin{pmatrix}

1 & 0 & 1 & 0 \\

0 & 1 & 0 & 1 \\

b_1 & d_1 &   b_2 & 0  \\

0& b_1 & d_2 & b_2
\end{pmatrix}
\begin{pmatrix}
\frac{u_{111}}{u_{11}}\\

\frac{u_{112}}{u_{11}}\\

\frac{u_{221}}{u_{22}}\\

\frac{u_{222}}{u_{22}}
\end{pmatrix} = \begin{pmatrix}
\alpha x_1\\

\alpha x_2\\

M_1-2\gamma \frac{x_1u_{11}}{u_{rr}}\\

M_2-2\gamma \frac{x_2u_{22}}{u_{rr}}

\end{pmatrix}
$$

and

$$b_i=\frac{u_{ii}}{\Delta u}+\gamma \frac{x_i^2u_{ii}}{u_{rr}}, \quad
\quad  d_i=2 \gamma u_{ii} \frac{x_1x_2}{u_{rr}}. $$

\medskip
\noindent The third order derivatives of $w$ solve the same system but with no dependence
on $M$ in the right hand
side vector (since the corresponding $M$ for $w$ is constant).

It is enough to show that the determinant of the third order derivatives coefficient matrix above is positive. This  determinant  is equal to

$$d_1d_2+(b_1-b_2)^2=4\gamma^2(\frac{x_1x_2}{u_{rr}})^2+ (b_1-b_2)^2$$
and can vanish only if one of the coordinates, say $x_2=0$, and $b_1=b_2$,
i.e.

$$u^2_{11}= \frac{1-\gamma}{1+ \gamma}=\beta-1.$$
This implies that $J(x)=J_0$ which is a contradiction. Thus, the
determinant is positive and the claim is proved.

\medskip

Since $H$ depends quadratically on $D^3u$ and $D^2u=D^2w$ at $x$,
the claim above implies that

\begin{equation*}
\begin{split}
|H(D^3u,D^2u,x)| &=|H(D^3u,D^2u,x)-H(D^3w,D^2u,x)| \\
& \le C(x,D^2u) (|\nabla M| +
|\nabla M|^2).
\end{split}
\end{equation*}

\noindent  Hence,  \eqref{eqn-mij}  implies that on the set where $J(x) \ne J_0$  there exists a smooth function
$C(x)$ depending on $u$ such that

$$ |u^{ij}M_{ij}| \le C(x) (|\nabla M| +
|\nabla M|^2).$$

\medskip
From the strong maximum principle, we conclude that $M$ cannot
have a local maximum or minimum in this set unless it is constant.
With this the Proposition is proved.

\end{proof}

Theorem \ref{t4} will follow from   the proposition below.

\begin{prop}{\label{p4.3}}
Suppose that $u$ is a  solution $u$ of (\ref{eq}), with $\alpha > -2$,  which satisfies
\begin{equation}{\label{4.rad}}
 c|x|^\beta \le u(x) \le C |x|^ \beta, \qquad \beta = 2+\alpha /2.
\end{equation}
Then the limit
 $$J_u(0):=\lim_{x \to 0} J_u(x)$$

\noindent exists. Moreover,  if for a sequence of $r_k \to 0$ the blow up solutions
$$v_{r_k}:=r_k^{- \beta} u (r_k x)$$
converge uniformly on compact sets to the solution $w$, then $w$ is homogenous of degree $\beta$ with $J_w=J_u(0)$.
\end{prop}

\begin{proof} From (\ref{4.rad}) we find that as $x \to 0$, $J_u(x)$ is
bounded away from
$0$ and $\infty$ by constants depending on $c$, $C$.
We will first show that $\lim_{x \to 0} J_u(x) =J(0)$ exists.

We may assume, without loss of
generality, that

$$ \limsup_{x \to 0} J_u(x):=k > J_0.$$

\noindent Let $x_i$ be a sequence of points for which $\limsup$ is achieved. The
blow up solutions $v_{r_i}$, $r_i=|x_i|$,  have a subsequence which converges uniformly
on compact sets of $\mathbb{R}^2$ to a solution $v$. Moreover, there
exists a point $y$ on the unit circle for which

$$J_v(y) = k \ge \limsup_{x \to 0} J_v,$$
hence, by Proposition \ref{p4.2}, $J_v$ is constant.

This argument also shows that if

$$J_u(z) \le k - \varepsilon \quad \mbox{then  $J_u(x) \le k -
\delta(\eps)$ on the circle $|x|=|z|$}.$$

\noindent Thus, if there exists a sequence of points $y_j \to 0$ with

$$ \lim_{y_j \to 0}  J_u(y_j) < k$$
then $J_u$ would have an interior maximum in  the annulus    $\{ x: \, |y_j| \le |x| \le
|y_{j'}|\}$ that contains one  of the points $x_i$ given above, a contradiction.
This shows that $\lim_{x \to 0} J_u(x)$ exists.

It remains to prove that if $J_v$ is constant,  then $v$ is
homogenous. It suffices to show that $D^2v$ is homogenous of degree
$\beta-2$, or more precisely that for each second derivative $v_{ij}$, we have
\begin{equation}\label{eqn-wii}
x \cdot \nabla v_{ij}= (\beta-2)\, v_{ij}.
\end{equation}

\noindent To this end,  for a fixed point $x$,  we
consider the homogenous solution $w$ with $D^2w(x)=D^2v(x)$. Since

$$\nabla J_v(x)=\nabla J_w(x)=0$$
the third derivatives of $v$ and $w$
solve the same system. We have seen in the proof of Proposition \ref{p4.2} that this system is solvable provided $ J_v  \neq J_0$.
Thus $D^3v(x)=D^3w(x)$ if $J_v \ne J_0$. Since  \eqref{eqn-wii}
is obviously true   for $w$, this implies that the equality holds for $u$ as well.

If $J_v=J_0$ we denote by $\Gamma$ the set where $D^2u(x)$ does
not coincide with the hessian of the radial solution. From the
proof of Proposition \ref{p4.2} we still obtain $D^3v(x)=D^3w(x)$
if $x\in \Gamma$, and by continuity \eqref{eqn-wii}
 holds for $x \in \bar \Gamma$. If $x$ is in the open set $
\bar \Gamma^c$, then $D^2 v$ coincides with $D^2 u_0$ and \eqref{eqn-wii}  is
again satisfied. This finishes the proof of the proposition.

\end{proof}

\noindent{\em Proof of Theorem \ref{t4}.} The proof of the theorem  readily  follows from Propositions \ref{p3.1}, \ref{p4.1} and \ref{p4.3}. \qed

\section{Proof of Theorem \ref{t3}}{\label{sec5}}

We consider the Dirichlet problem

\begin{equation}\label{dc}
\begin{cases}
\det D^2 u=|x|^\alpha  \quad &\text{in $B_1$}\\
u=u_0- \eps \cos (2 \theta)  \quad &\text{on $\partial B_1$}
\end{cases}
\end{equation}
in the range of exponents $\alpha >0$. Here
$$u_0(x)=c_\alpha\, |x|^\beta, \qquad \beta = 2+\alpha/2 $$
denotes  the radial solution of  the equation, i.e,  $\det D^2 u_0=|x|^\alpha$.
We  write the solution as
\begin{equation}\label{eqn-ve}
u=u_0- \eps v.
\end{equation}
Heuristically, is $\epsilon$ is small $v$ satisfies the linearized equation
at $u_0$, namely

$$(D^2 u_0)^{-1} : D^2 v=0,$$
where we use the notation $A:B = \sum_{ij} a_{ij}\, b_{ij}$ for
the Frobenius inner product between two $n\times n$ matrices $A$
and $B$.

At any point $x_0 \in B_1$, we denote by $\nu$ and $\tau$ the unit
normal (radial) and unit tangential direction, respectively, to
the circle $|x|=|x_0|$ at $x_0$. In $(\nu,\tau)$ coordinates,
$$D^2 u_0 = c_\alpha r^{\beta-2} \begin{pmatrix}
  \beta (\beta-1)  & 0 \\
  0 & \beta
\end{pmatrix}$$
hence, $v$ satisfies
the equation
$$v_{\nu\nu}+ (\beta -1) v_{\tau\tau}=0.$$
Solving this equation with boundary data $v=\cos (2 \theta)$ we obtain
the solution
$$v=r^\rho \cos (2 \theta)$$
with
$$\rho (\rho -1) + (\beta -1)(\rho -4)=0.$$
Solving the quadratic equation with respect to $\rho$
gives
$$\rho = \frac{2-\beta \pm \sqrt{\beta^2 +12\, \beta - 12}}{2}.$$
Since $\beta:=2+\alpha/2>2$ the only acceptable solution is
$$\rho = \frac{2-\beta + \sqrt{\beta^2 +12\, \beta - 12}}{2}$$
and it satisfies
\begin{equation}\label{eqn-br}
2 < \rho < \beta
\end{equation}
which suggests that close to the origin the perturbation term
$\eps v$ dominates $u_0$.

\medskip

We wish to show that the solution $u$ of the Dirichlet problem
\eqref{dc} admits at the origin the non-radial behavior
\eqref{nonrad}, if $\eps \leq \eps_0$, with $\eps_0$ sufficiently
small. We will argue by contradiction. Assume,
 that $u$ has the radial behavior

$$c_0\, |x|^ \beta \le u(x) \le C_0 \, |x|^ \beta$$
with $c_0$, $C_0$ universal constants. By rescaling, we deduce that

$$c \, I \le |x|^ {2- \beta } D^2 u(x) \le C \, I$$
with $I$ denoting the identity matrix.

\medskip
The function $v$ which is defined by \eqref{eqn-ve} satisfies

$$|v| \le 1, \qquad v=\cos (2 \theta) \mbox{ on $\partial B_1$}$$
and solves the equation
$$a^{ij}v_{ij}=0$$
with
$$ A= (a^{ij})=
\int_0^1(tD^2u_0+(1-t)D^2u)^{-1}dt=\int_0^1(D^2u_0+\eps(t-1)D^2v)^{-1}dt.$$
Hence
\begin{equation}\label{eqn-Ar}
c \, I \le r^{ \beta -2} A \le C\, I.
\end{equation}

\noindent The solution  $u$ has bounded third order derivatives in $B_1
\setminus B_{1/2}$, thus

$$|D^2v(x)| \le C \, \|v\|_{L^\infty} \le C \qquad \mbox{
in $B_1\setminus B_{1/2}$}.$$

\noindent By rescaling we obtain the bound

$$|D^2v(x)| \le C |x|^{-2}.$$
From this we find that

$$r^{\beta -2} |A-D^2u_0^{-1}| \le C \eps r^{-\beta}$$
hence, $v$ satisfies the Dirichlet problem

\begin{equation}\label{dv}
\begin{cases}
f^{ij}v_{ij}=0  \quad &\text{in $B_1$}\\
v=\cos (2 \theta) \quad &\text{on $\partial B_1$}
\end{cases}
\end{equation}
with
$$F:=c\, r^{\beta -2}\, A $$
hence,  by \eqref{eqn-Ar},
$$c\, I \le F \le C\, I \quad \mbox{and} \quad
|F-F_0| \le C \eps r^{-\beta}$$
with
$$F_0:=\nu \otimes \nu + (\beta-1) \,
\tau \otimes \tau.$$
(As before, we denote  by $\nu$ and $\tau$ the unit normal (radial) and unit tangential directions,  to  the circle $|x|=|x_0|$ at each point $x_0\in B_1$).

\noindent Also,
$$ |v| \le 1, \quad  \mbox{ on $B_1$}.$$

\noindent From the definitions of $A$ and $F$ we also obtain

\begin{equation}{\label{5.1}}
\|\nabla(F-F_0)\| \le C(r_0)\,  \eps \qquad \mbox{ for $|x| \ge r_0$.}
\end{equation}
Set

$$w:=r^\rho \cos (2 \theta).$$
Then, $w$ satisfies the equation
$$ F_0:D^2w =0$$
thus, we have

$$|f^{ij}\, w_{ij}| \le C r^{\rho -2} \min \{\eps r^{- \beta},1 \}.$$
Applying the Aleksandrov maximum principle on $v-w$ (see  Theorem 9.1 in
\cite{GT}),  we find that

$$|v-w| \le C \, \eps^\delta$$
and therefore (see (\ref{5.1}))

\begin{equation}{\label{secbd}}
|D^2v-D^2w| \le C'(r_0) \, \eps^\delta,  \qquad \mbox{ for $|x| \ge r_0$}.
\end{equation}

\medskip

We next  compute

$$M_u(x):=  \log(\Delta u) + \gamma \log (r^2u_{rr}), \qquad \gamma:=\frac 2\beta -1$$

\noindent in terms of $M_{u_0}$, for $|x| \ge r_0$, with $r_0$ small, fixed.
We recall that $M_{u_0}$ is constant in $x$. Since  $u=u_0-\eps\, v$,
we find that

$$ M_u(x) = M_{u_0} - \eps \left (  \frac{\Delta v}{\Delta u_0} + \gamma
\, \frac{v_{rr}}{u_{0,rr}} \right ) - \frac{\eps^2}{2} \left (
(\frac{\Delta v}{\Delta u_0})^2 + \gamma \,
(\frac{v_{rr}}{u_{0,rr}})^2 \right ) + O(\eps^3). $$ Because
$$\det D^2u = \det D^2 u_0$$
the function $v$ satisfies the equation
$$ -u_{0,rr}\, v_{\tau \tau}-u_{0,\tau \tau} \, v_{rr} +\eps \det D^2 v=0$$
or equivalently (since $u_0(r)= c_\alpha\, r^\beta$)
$$ v_{rr}+(\beta -1)\, v_{\tau \tau}=\eps\, \frac{ r^{2- \beta}}{c_\alpha\, \beta}\,  \det D^2 v.$$
The last equality implies that
$$ \frac{\Delta v}{\Delta u_0} + \gamma
\frac{v_{rr}}{u_{0,rr}}=\eps \frac{r^{2(2- \beta)}}{c_\alpha^2 \,
\beta^3\, (\beta-1)} \det D^2v,$$

\noindent and also  that
\begin{equation*}
\begin{split}(\frac{\Delta v}{\Delta u_0})^2 + \gamma\,
(\frac{v_{rr}}{u_{0,rr}})^2&=(1+ \frac{1}{\gamma})\, (\frac{\Delta
v}{\Delta u_0})^2 + O(\eps)\\
&= -\frac{2r^{2(2- \beta)}}{ c_\alpha^2 \, \beta^2\, (\beta
-2)}(\Delta v)^2+ O(\eps).
\end{split}
\end{equation*}

\noindent From (\ref{secbd}) and the above we conclude  that

$$M_u(x)=  M_{u_0}  + \eps^2 r^{2(2-\beta)} \, [ a_1\, (-\det D^2w)+a_2\,
(\Delta w)^2 ] +O(\eps^{2+\delta})$$
for $|x| \geq r_0$, with $O(\eps^{2+\delta})$ depending on $r_0$.
The constants $a_1$ and $a_2$ are given by
$$a_1 = \frac{1}{c_\alpha^2 \, \beta^3\, (\beta-1)} \quad \mbox{and} \quad
a_2 = \frac{2}{c_\alpha^2 \, \beta^2\, (\beta -2)}.$$
\medskip

\noindent We recall that $w(r,\theta)=r^\rho \, \cos(2\theta)$.
Then, a direct computation shows that each term in the square
brackets above is positive. Thus the $\eps^2$ term is positive and
homogeneous of degree $2\, (\rho-\beta)$, with  $\rho < \beta$ (as
shown in \eqref{eqn-br}). We  conclude from Proposition \ref{p4.2}
that
$$ \lim_{x \to 0} M_u(x)  > M_{u_0}.$$

\medskip

\noindent Hence, from Proposition \ref{p4.3},   the blowup limit
of $u$ at the origin cannot be $u_0$. On the other hand, from the
symmetry of the boundary data for $u$ we conclude that the
function $v-v(0)$ has exactly two disconnected components where it
is positive (or negative). Thus the blowup limit at the origin for
$u$ has period $\pi$ on the unit circle which contradicts
Proposition \ref{p4.1}.

\qed

\section{Proof of Theorem \ref{t1}}{\label{sec2}}

In this final section we will present the last steps of the proof of Theorem \ref{t1}.
We distinguish the two different cases of behavior at the origin,
\eqref{rad} and \eqref{nonrad}.

\medskip

\noindent{\em Case 1: Radial Behavior.} We will show that  solutions
of \eqref{eq}  with  the radial behavior
(\ref{rad}) are $C^{2, \frac \alpha 2}$.

We begin by observing that solutions of \eqref{eq} satisfy, in $B_1
\setminus B_{1/2}$, the estimate

\begin{equation}\label{eqn-c11}
\|D^2u\|_{C^{0,1}(B_1 \setminus B_{1/2})}\le C(\alpha)
\end{equation}
provided that
\begin{equation}\label{eqn-c12}
c(\alpha) \, |x|^{2+ \frac \alpha 2} \leq u(x)  \leq C(\alpha)\,  |x|^{2+ \frac \alpha 2}.
\end{equation}

For any $r >0$, the rescaled functions
\begin{equation}{\label{2.01}}
u^r(x):=r^{-2- \frac \alpha 2}\, u(rx)
\end{equation}
solve the equation \eqref{eq}. Since $u$ has the radial behavior \eqref{rad} at the origin,
each function $u^r$ satisfies \eqref{eqn-c12}.
Hence, applying \eqref{eqn-c11} to $u^r$,  we obtain for
$x,y \in B_1 \setminus B_{1/2}$ the estimates

$$|D^2u(rx)-D^2u(ry)|\le r^{\frac \alpha 2}|x-y|, \qquad |D^2u(rx)|\le
C\, r^{\frac \alpha  2}.$$
The above estimates, readily imply that  $u \in C^{2, \frac \alpha 2}$.

\medskip

\noindent{\em Case 2: Non-radial Behavior.} In the rest of the
section we will  show  that solutions of \eqref{eq} which  satisfy
the nonradial behavior (\ref{nonrad}) are also of class
$C^{2,\delta}$, for some $\delta >0$. The idea is simple: we
approximate $u$ with quadratic polynomials in the $x_2$ direction.
However, the proof is quite technical.

In order to simplify the constants, we assume that $u$ solves the equation

\begin{equation}{\label{eq1}}
\det D^2u =2(2+\alpha)(1+\alpha)\, |x|^\alpha
\end{equation}
instead of \eqref{eq} and (after rescaling) that

\begin{equation}{\label{exp}}
u(x)=|x_1|^{2+\alpha}+x_2^2+O
\left((|x_1|^{2+\alpha}+x_2^2)^{1+\delta}\right), \qquad \mbox{as} \, |x| \to 0.
\end{equation}

\noindent From now on, we  will  denote points in $\mathbb{R}^2$ with capital letters
$$X=(x_1,x_2).$$

\medskip

The H\"older continuity of the second order derivatives of $u$  follows easily
from the following proposition.

\begin{prop}{\label{p2.1}}
Let $\lambda>0$ be small and
$$Y \in \Omega_\lambda:= \{ \lambda \le |x_1|^{2+\alpha} + x_2^2 \le 2
\lambda \}.$$
Then, there exist $C$, $\mu$ universal constants such that in
$B:=B(Y,\lambda^{1+\alpha})$, we have

$$\|D^2u\|_{C^\mu(B)} \le C \quad \mbox{and} \quad  \|D^2u-D^2u(0)\|_{L^\infty(B)} \le \lambda^{\mu}.$$
\end{prop}

\medskip

We will show that in the sections

$$S_{{X_0},t}:=\{\, X: \, u(X) < u(X_0) + \nabla u(X_0) \cdot (X-X_0) +t \}.$$

\noindent of $u$ at the point
$$X_0=(0,x_0), \quad |x_0|\le 2 \lambda^{1/2}$$
we can approximate $u$ by quadratic polynomials of opening
$2$ on vertical segments. We begin by making  the following definition.

\begin{defn} We say that
$$u \in Q(e,\eps, \Omega)$$
if for any vertical segment $l \subset \Omega$ of length less than $e$,
there exists a quadratic polynomial $P_{x_1,l}(x_2)$ of opening $2$, namely

$$P_{x_1,l}(x_2)=x_2^2+p(x_1,l)\, x_2+r(x_1,l)$$
such that

$$\left |u(x_1,x_2)-P_{x_1,l}(x_2) \right| \le \eps e^{2} \quad \mbox {on $l$}.$$
\end{defn}

\noindent Notice that for  $c<1$ we  have
$$Q(e,\eps, \Omega) \subset Q(ce,c^{-2}\eps, \Omega).$$

\medskip

{\em The plan of the proof is as follows:} We prove Proposition \ref{p2.1} for points $Y \in S_{X_0,t}$, with $t
\le \lambda$. We first show that $u$ belongs to some appropriate
$Q$ classes   and distinguish two cases;  one when $t \ge
\lambda^{\frac \alpha 2 +1 -\delta_1}$ for some fixed
$\delta_1>0$, and the other when $t=\lambda^{\frac \alpha 2 +1
-\delta_1}$. In the first case we use the same method as in Lemma
\ref{l1} and approximate the right hand side $|f(X)|^{\alpha /2}$
of the rescaled Monge-Amp\'ere equation with $|x_1|^\alpha$ (see
Lemma \ref{l2.1}). In the second case we approximate $f(X)$ with a
more general polynomial $x_1^2 + p x_1 + q$ and obtain a better
approximation ($Q$ class) for $u$ (Lemma \ref{l2.2}).

The H\"older estimates for points $Y \in S_{X_0,t}$, $|x_0|\ge
\lambda^{1/2}$ are obtained in appropriate sections $S_{Y,\sigma}$
in which all the values of $|x|$ are comparable. In these sections
the Monge-Amp\'ere equation is nondegenerate and the classical
estimates apply. To obtain the appropriate  section  $S_{Y,\sigma}$ we
distinguish two cases,  depending on the distance from $Y$ to the
$x_2$ axis. If $|y_1|\ge \lambda^{1/2}$,  then we take $\sigma$ so
that $S_{Y,\sigma}$ is at distance greater than $|y_1|/2$ from the
$x_2$ axis (Lemma \ref{l2.3}). If $|y_1|\le \lambda^{1/2}$, then we
take $\sigma=\lambda^{\frac{2+\alpha}{2}}$ and $S_{Y,\sigma}$ is
close enough to the $x_2$ axis so that all its points are at
distance comparable to $\lambda^{1/2}$ from the origin (Lemma
\ref{l2.4}).

\smallskip

In what follows we will denote by $A_t$, $D_t$ the matrices

$$A_t=\begin{pmatrix}
  a_{11} & 0 \\
  a_{21} & a_{22}
\end{pmatrix} , \qquad  D_{t}=\begin{pmatrix}
  t^\frac{1}{2+\alpha} & 0 \\
  0 & t^{\frac{1}{2}}
\end{pmatrix}.$$

\begin{lem}{\label{l2.1}} Let $X_0=(0,x_0)$ with $ |x_0|\le 2 \lambda^{1/2}$, $0 < \lambda <1$.
Then, for any $\delta_1>0$ and
\begin{equation}{\label{2.1}}
\lambda^{\frac{\alpha}{2}+1-\delta_1} \le t \le \lambda
\end{equation}
there exists a small $\delta_2>0$, depending on $\delta_1$, such that
\begin{equation}{\label{2.2}}
S_{X_0, t}-X_0 \in A_t D_t (\Gamma \pm t^{\delta_2})
\end{equation}
with
\begin{equation}\label{eqn-At}
|A_t-I|\le t^{\delta_2}.
\end{equation}
Moreover,
$$u\in Q(t^{\frac 12}, \lambda^{\delta_2}, S_{X_0, t}).$$
\end{lem}

\begin{proof}  We begin by observing that if $t=\lambda$,  then the
conclusion of
the lemma follows from the expansion (\ref{exp}) with matrix
$A_t=I$. We will show by induction, using at each step  the approximation lemma \ref{l1},
that (\ref{2.2}) and \eqref{eqn-At} hold for every $t=\lambda\, t_0^k$,  $k \in \mathbb N$,  which satisfies \eqref{2.1}.

Assume that  (\ref{2.2}) and \eqref{eqn-At}  hold for some $t=\lambda\, t_0^k$ satisfying (\ref{2.1}),  with
$A_t$ bounded and $a_{t, 11}$ bounded from below.
Consider the rescaling
\begin{equation}{\label{2.3}}
v(X):=\frac{1}{t} \, ( u(X_0+A_t D_t X) - u(X_0) - \nabla u(X_0)
\, (A_t D_t X) \, ).
\end{equation}
Since $u$ satisfies \eqref{eq1}, the function $v$ satisfies the equation
\begin{equation}\label{eqn-v1}
\det D^2 v =2(2+\alpha)(1+\alpha) a_{11}^2\, a_{22}^2 \, t^{-\frac {\alpha}{2+\alpha}} \, |X_0+A_t D_t X|^\alpha.
\end{equation}
Since
\begin{equation}\label{eqn-f1}
|X_0+A_t D_t X|^2 = (t^{\frac 1{2+\alpha}}\, a_{11} x_1)^2 + (t^{\frac 1{2+\alpha}} a_{12} x_1 +  t^{\frac 12} a_{22} x_2 + x_0)^2
\end{equation}

\medskip
\noindent
and $|x_0| \leq 2\, \lambda^{1/2}$, we conclude from the above that $v$ satisfies

\begin{equation}{\label{2.4}}
\det D^2 v = c \, |f(X)|^{\frac{\alpha}{2}}, \qquad S^v_{0,1} \in \Gamma \pm
t^{\delta_2}
\end{equation}
with
$$|f(X)-x_1^2| \le C \,
\left(\lambda^\frac{1}{2}t^{-\frac{1}{2+\alpha}}+
t^\frac{\alpha}{2(2+\alpha)}\right)
\le t^{\frac{\delta_1}{2(2+\alpha)}}.$$
Notice that the last inequality holds if \eqref{2.1} is satisfied.

Lemma \ref{l1} with $\eps=t^{\delta'}$,
$\delta'(\delta_1,\alpha)>0$ small, yields

$$S^u_{X_0, t_0 t}-X_0 \in A_{t_0 t} D_{t_0 t} (\Gamma \pm (t_0
t)^{\delta_2})$$
with
$$A_{t_0 t}=A_t E_{t}, \quad |E_t-I| \le
C\, t^{\delta_2}.$$ Thus, (\ref{2.2}) and \eqref{eqn-At} hold for $t'=t\, t_0$.
If $t' \leq \lambda^{\frac \alpha 2+ 1 - \delta_1}$ we stop, otherwise we continue the induction.

From (\ref{2.4}) we find that
\begin{equation}{\label{vbd}}
\left|v-(|x_1|^{2+\alpha}+x_2^2)\right|\le C t^{\delta_2} \qquad \mbox{in
$S^v_{0,1}$}
\end{equation}
which together with (\ref{2.3}) and  \eqref{eqn-At}, yields to
$$u \in Q(t^{\frac 12},C \, \lambda^{\delta_2}, S^u_{X_0,t}).$$

\noindent The lemma is proved by replacing $\delta_2$ with $\delta_2 /2$.
\end{proof}

\medskip

We will next examine closer the borderline  case $t=\lambda^{\frac{\alpha}{2}+1-\delta_1}$
and show the  better approximation \eqref{eqn-better} of $u$ by  quadratic  polynomials in the
$x_2$ variable. We begin by observing that the conclusion of the previous lemma implies
that
$$S_{X_0, t}-X_0 \in A_t D_t (\Gamma \pm
\lambda^{\delta_2}),  \qquad |A_t-I|\le \lambda^{\delta_2}
$$
for all $
\lambda^{\frac{\alpha}{2}+1-\delta_1}\leq t \leq \lambda$.

\medskip
\begin{lem}{\label{l2.2}}
Assume that for $t=\lambda^{\frac{\alpha}{2}+1-\delta_1}$ and
$\delta_2 \ll \delta_1$, we have

\begin{equation}{\label{2.5}}
S_{X_0, t}-X_0 \in A_t D_t (\Gamma \pm
\lambda^{\delta_2}), \quad \quad |A_t-I|\le \lambda^{\delta_2}.
\end{equation}
Then if $\delta_1$ is small, universal, we have

\begin{equation}\label{eqn-better}
u \in Q(e, C\, \lambda^{\delta_2}, S_{X_0, \frac t2}), \quad \quad
\mbox{for all $e$ with $\lambda^\frac{2+\alpha}{4} \le e \le
t^{1/2}$}.
\end{equation}

\end{lem}

\begin{proof} Let $v$ be the re-scaling  defined in (\ref{2.3}). It follows from \eqref{eqn-v1},
\eqref{eqn-f1} and \eqref{2.5} that $v$ satisfies

$$\det D^2 v = c \, f(X)^\frac{\alpha}{2}, \quad \quad S^v_{0,1} \in \Gamma
\pm \lambda^{\delta_2}$$
with
$$|f(X)-x_1^2-p\, x_1-q| \le t^\frac{\alpha}{2(2+\alpha)}, \quad \quad
|p|,|q|\le \lambda^\frac{\delta_1}{2+\alpha}$$
thus
$$\left |f(X)^\frac{\alpha}{2}-(x_1^2+px_1+q)^\frac{\alpha}{2} \right|\le
\eps:=t^{\delta_0(\alpha)}, \quad \quad
\delta_0(\alpha)=\frac{\alpha \min
\{\alpha,2\}}{4(2+\alpha)}.$$

\noindent Similarly as  in the proof of Lemma \ref{l1} we define
the function $w$ as the solution to

$$\det D^2 w=c\, (x_1^2+px_1+q)^\frac{\alpha}{2}, \quad \mbox{$w=1$ on
$\partial S^v_{0,1}$}$$

\noindent and obtain (see (\ref{vbd})) that

$$|v- w| \le C \eps ^\frac{1}{2}=C t^\frac{\delta_0(\alpha)}{2}$$
and
$$\left |w-(|x_1|^{2+\alpha}+x_2^2) \right | \le C \lambda^{\delta_2}.$$

By considering the partial Legendre transform $w^*$, one can deduce
from the last inequality, the bounds on $|p|$, $|q|$ and Lemma \ref{lineq} that

$$|w_{22}-2| \le C \lambda^{\delta_2} \quad \mbox{in $S^v_{0, 1/2}$}.$$
This implies that

$$w \in Q(e,C \lambda^{\delta_2}, S^v_{0, 1/2} ),  \quad \quad \mbox{for any
$e$}$$
hence

$$v \in  Q(e,C \lambda^{\delta_2}, S^v_{0, 1/2}),   \quad \quad \mbox{for $e \ge
t^\frac{\delta_0(\alpha)}{8}$}.$$

\medskip
\noindent Then,  similarly  as at the end of the proof of  the previous lemma, we obtain that

$$u \in Q( t^{1/2} e, C \lambda^{\delta_2}, S^u_{0,t/2}),
\quad \quad \mbox{for $e \ge
t^\frac{\delta_0(\alpha)}{8}$}$$

\medskip
\noindent from which the lemma follows,  since

$$t^\frac{1}{2}t^\frac{\delta_0(\alpha)}{8} \le
\lambda^\frac{2+\alpha}{4}$$ for
$\delta_1$ small, universal (depending only on $\alpha$).
\end{proof}

The   next lemma   proves Proposition \ref{p2.1} for a point $Y \in
S_{X_0,\lambda}$ at distance greater than $\lambda^{1/2}$ from the $x_2$
axis, assuming  the conclusions of lemmas \ref{l2.1} and \ref{l2.2}.

\begin{lem}{\label{l2.3}}
Assume that for $\lambda^{\frac{\alpha}{2}+1-\delta_1} \le t \le \lambda$, we have

\begin{equation}{\label{2.6}}
S_{X_0, t}-X_0 \in A_t D_t (\Gamma \pm
\lambda^{\delta_2}), \quad \quad |A_t-I|\le \lambda^{\delta_2}
\end{equation}
and

$$u \in Q(e,C\, \lambda^{\delta_2}, S_{X_0, \frac t2}) \quad \quad
\mbox{for some $e$,} \quad \lambda^\frac{2+\alpha}{4} \le e \le
t^\frac{1}{2}.$$
If
\medskip
$$Y=(y_1,y_2) \in S_{X_0, \frac t3}, \quad \quad 1 \le |y_1|
e^{-\frac{2}{2+\alpha}}
\le 2$$

\medskip

\noindent then $D^2u$ is H\"older continuous in the ball $B:=B(Y, \lambda^{1+\alpha})$, and for
some constant $0 < \beta < 1$, it satisfies
\begin{equation}\label{eqn-w2d}
\|D^2 u\|_{C^{0,\beta}(B)} \le C \quad \quad \mbox{and} \qquad
|D^2u(Y)-D^2u(0)| \le C\,  \lambda^{\beta}.
\end{equation}
\end{lem}

\begin{proof} Consider the section $S^u_{Y,ce^2}$ for a small constant $c$.
By Theorem \ref{caf1} there  exists a matrix

$$F:=\begin{pmatrix}
  a & 0 \\
  d & b
\end{pmatrix}, \quad \quad a, b > 0$$
such that

\begin{equation}\label{eqn-F}
FB_{1/{C_0}} \subset S^u_{Y,ce^2}-Y \subset FB_{1}, \qquad
\mbox{$C_0(\alpha)>0$ universal}.
\end{equation}
Using the assumptions of the lemma and \eqref{eqn-F} we will derive bounds
on the coefficients of the matrix $F$.
Clearly,
$$\nu:=\frac{c^{1/2}e}{b}$$
satisfies the bound
\begin{equation}{\label{bbd}}
\frac{1}{2C_0} \le \frac{c^{1/2}e}{b} \le 2.
\end{equation}

\noindent Since $e \le t^\frac{1}{2}$, the corresponding section for the
rescaling $v$ (see (\ref{2.3}),   (\ref{vbd})) satisfies

$$S^v_{\tilde{Y},\frac{ce^2}t} \subset \{|x_1|^{2+\alpha}+x_2^2 \le 3/4 \}$$
or more precisely

$$S^v_{\tilde{Y},\frac{ce^2}t}-\tilde{Y} \subset
\left(
(e^2/t)^\frac{1}{2+\alpha}+\lambda^\frac{\delta_2}{2(2+\alpha)}\right
) B_1$$
thus,

$$D_t^{-1}A_t^{-1}F B_{1/C_0} \subset
\left(
e^\frac{2}{2+\alpha}t^{-\frac{1}{2+\alpha}}+\lambda^{\delta_3}
\right) B_1.$$
The last inclusion implies the estimate

\begin{equation}{\label{dbd}}
|d| \le 2C_0\,  ( e^\frac{2}{2+\alpha}t^\frac{\alpha}{2(2+\alpha)} +
\lambda^{\delta_3}t^\frac{1}{2} + a \lambda^{\delta_3}) \le
4C_0\, (e^\frac{2}{2+\alpha}+a)\, \lambda^{\delta_3}.
\end{equation}

\medskip
\noindent The rescaling

$$w(x):=\frac{1}{b^2}u(Y+Fx)$$
satisfies

\begin{equation}{\label{2.7}}
\det D^2 w = \frac{a^2}{b^2}f(x)^\frac{\alpha}{2}, \qquad B_{1/C_0} \subset
S^w_{0, \nu^2} \subset B_1
\end{equation}
with
\begin{equation}{\label{2.8}}
f(x):=(y_1+ax_1)^2+(y_2+ d x_1+ b x_2)^2
\end{equation}
and
\begin{equation}{\label{2.9}}
|w-P'_{x_1}(x_2)| \le \lambda^ {\delta_3},  \qquad \mbox{in $S^w_{0,
\nu^2}$}.
\end{equation}

\noindent We claim that if $c$ is chosen small, universal, then

\begin{equation}{\label{abd}}
2a \le e^\frac{2}{2+\alpha} \le |y_1|.
\end{equation}
Indeed, otherwise from \eqref{2.7}, we deduce that

$$\det D^2w \ge a^{2+\alpha} b^{-2}\left(x_1+\frac{y_1}{a}
\right)^\alpha $$
with

$$(2a)^{2+\alpha}b^{-2} \ge e^2 b^{-2} \ge c^{-1}  \nu^2$$
and for small $c$ we contradict $B_{1/C_0} \subset S^w_{0, \nu^2}$,  since
$\nu$ is bounded.

\medskip
From (\ref{bbd}), (\ref{dbd}), (\ref{abd}) and $|y_2| \le 4
\lambda^{1/2}$ we obtain that ${f(x)} /{y_1^2}$ is bounded away
from $0$ and $\infty$ by universal constants, and also its
derivatives are bounded by universal constants. From (\ref{2.7})
we find that

$$ c_1 \le \frac{a^2|y_1|^\alpha}{b^2} \le C_1$$
which implies that  $a^{2+\alpha}$, $|y_1|^{2+\alpha}$, $b^2$, and $e^2$ are all
comparable. Moreover, using also (\ref{2.9}), we have
\begin{equation}{\label{2.10}}
\|D^2w\|_{C^{0,1}} \le C, \quad |w_{22}-2| \le \lambda^{\delta_4} \quad
\mbox{in
$S_{0, \nu ^2}/2$}.
\end{equation}
Hence

\begin{equation}{\label{2.11}}
|w_{22}(x)-w_{22}(y)| \le C\, \lambda^{{\delta_4}/2}|x-y|^{1/2} \quad \quad \mbox{
for $x,y \in S_{0, \nu^2}/2$}.
\end{equation}

\medskip
\noindent Also, we  have

$$D^2u(Y+Fx)=b^2(F^{-1})^TD^2w(x)\, F^{-1}$$
with

$$b \, F^{-1}= \begin{pmatrix}
  b /a & 0 \\
 - d /a & 1
\end{pmatrix}
=\begin{pmatrix}
  0 & 0 \\
  0 & 1
\end{pmatrix} +O(\lambda^{\delta_3})$$
which together with (\ref{2.10}) implies the second part of the
conclusion \eqref{eqn-w2d}.
\medskip
Finally, since

$$|Fx|\ge  \frac{b|x|}2 \ge \lambda ^{1+\alpha}|x|$$
we obtain from (\ref{2.10}) and  (\ref{2.11}) the estimate

$$|D^2u(Y+Fx)-D^2u(Y+Fy)| \le C\lambda^{{\delta_4}/2}|x-y|^{1/2} \le C\,
|Fx-Fy|^{\beta}.$$
This finishes the proof of the lemma.
\end{proof}

\medskip
The next lemma  proves H\"older continuity when $Y$ is
$\lambda^{1/2}$ close to the $x_2$ axis.

\begin{lem}{\label{l2.4}}
Assume that $(\ref{2.6})$ holds for
$t=\lambda^{\frac{\alpha}{2}+1-\delta_1}$,

$$u \in Q(e,\lambda^{\delta_2}, S_{X_0, \frac t2})  \quad \quad
\mbox{for} \quad e=\lambda^\frac{2+\alpha}{4}$$
and

$$|x_0| \ge \lambda^\frac{1}{2}/2 , \quad Y\in S_{X_0,\frac t3},
\quad |y_1| \le
e^\frac{2}{2+\alpha}.$$

\noindent Then, the conclusion of Lemma $\ref{l2.3}$ still holds.
\end{lem}

\begin{proof}
The proof is very similar to that  of Lemma \ref{l2.3}. The
only difference  is that now the second term of $f$ in (\ref{2.8})
dominates the sum.

Indeed, since $ \lambda^{1/2} \ge |y_1| $ and $|y_2| \ge
\lambda^{1/2}/4$, the function
${f(x)}/ {y_2^2}$
is bounded away from $0$ and $\infty$ by
universal constants, and also its derivatives are bounded by universal
constants. Hence, $a^{2+\alpha}$, $y_2^{2+\alpha}$, $b^2$ and $e^2$ are
all comparable and the rest of the proof is the same.
\end{proof}

\medskip

\begin{proof} [Proof of Proposition \ref{p2.1}]

For $Y \in \Omega_\lambda$ we consider the section $S^u_{Y,\sigma}$ that
becomes tangent to the $x_2$ axis at $X_0=(0,x_0)$. Since $|x|^\alpha dx$ is doubling,
there exists $C_1$ universal such that
$$Y \in S^u_{X_0,t/3}, \quad |x_0| \le 2\lambda^{1/2}, \quad t:=C_1\sigma \le C_2
\lambda.$$

\noindent We distinguish the following three cases:

\begin{enumerate}[i.]

\item If $t \ge t_0:=\lambda^{\alpha/2 +1 - \delta_1}$, then  the proposition follows from Lemmas
\ref{l2.1} and \ref{l2.3} with
 $$ e=|y_1|^\frac{2+\alpha}{2} \ge c_1 t^{1/2}.$$

\item If $t \le t_0$ and $|y_1| \ge \lambda^{1/2}$,  then we
apply Lemmas \ref{l2.2} and \ref{l2.3} for $S_{X_0,t_0}$ with $e$
defined as above.

\item If $t \le t_0$ and $|y_1| \le \lambda^{1/2}$,
then we  apply Lemma \ref{l2.2} and Lemma \ref{l2.4}. We remark that the
hypothesis $|x_0| \ge \lambda^{1/2} /2$ is satisfied because $Y \in \Omega_\lambda$.
\end{enumerate}
\end{proof}

\end{document}